\documentclass[12pt,a4paper,twoside]{amsart}

\usepackage{mathrsfs}          
\usepackage[all]{xy}    
\usepackage[latin1]{inputenc}
\usepackage{amscd}
\usepackage{latexsym}
\usepackage{multicol}
\usepackage{amsmath}
\usepackage{amssymb}
\usepackage{amsthm}

\theoremstyle{plain}
\newtheorem{Theorem}{Theorem}[section]
\newtheorem{lemma}[Theorem]{Lemma}
\newtheorem{corollary}[Theorem]{Corollary}
\newtheorem{proposition}[Theorem]{Proposition}

\newtheorem{conjecture2*}{Conjecture}
\newtheorem{remark2*}[conjecture2*]{Remark}
\newtheorem{Theorem2*}[conjecture2*]{Theorem}
\newtheorem{proposition2*}[conjecture2*]{Proposition}
\newtheorem{claim2*}[conjecture2*]{Claim}
\newtheorem{example2*}[conjecture2*]{Example}
\newtheorem{lemma2*}[conjecture2*]{Lemma}

\theoremstyle{definition}

\newtheorem{example}[Theorem]{Example}

\newtheorem{definition}[Theorem]{Definition}
\newtheorem{remark}[Theorem]{Remark}

\newtheorem{pkt}[Theorem]{}
\newtheorem{construction}[Theorem]{Construction}

\newcommand{\sB}{{\mathcal B}}
\newcommand{\sC}{{\mathcal C}}

\newcommand{\sE}{{\mathcal E}}

\newcommand{\sH}{{\mathcal H}}

\newcommand{\sO}{{\mathcal O}}

\newcommand{\sV}{{\mathcal V}}
\newcommand{\sW}{{\mathcal W}}

\newcommand{\sX}{{\mathcal X}}
\newcommand{\sY}{{\mathcal Y}}
\newcommand{\sZ}{{\mathcal Z}}
\newcommand{\A}{{\mathbb A}}
\newcommand{\B}{{\mathbb B}}
\newcommand{\C}{{\mathbb C}}

\newcommand{\G}{{\mathbb G}}

\newcommand{\bP}{{\mathbb P}}
\newcommand{\Q}{{\mathbb Q}}
\newcommand{\R}{{\mathbb R}}
\newcommand{\BS}{{\mathbb S}}

\newcommand{\Z}{{\mathbb Z}}
\newcommand{\id}{{\rm id}}

\newcommand{\Hg}{{\rm Hg}}
\newcommand{\MT}{{\rm MT}}
\newcommand{\SL}{{\rm SL}}
\newcommand{\GL}{{\rm GL}}

\newcommand{\Mon}{{\rm Mon}}
\newcommand{\Sp}{{\rm Sp}}

\newcommand{\SU}{{\rm SU}}
\newcommand{\SO}{{\rm SO}}
\newcommand{\PU}{{\rm PU}}
\newcommand{\U}{{\rm U}}
\newcommand{\GSp}{{\rm GSp}}

\newcommand{\ad}{{\rm ad}}
\newcommand{\der}{{\rm der}}

\newcommand{\diag}{{\rm diag}}
\newcommand{\fg}{\mathfrak{g}}
\newcommand{\fh}{\mathfrak{h}}

\newcommand{\fsu}{\mathfrak{su}}

\title{Calabi-Yau manifolds and generic Hodge groups}
\author{Jan Christian Rohde}
\address{GRK 1463 / Institut f\"ur Algebraische Geometrie\\ Leibniz Universit\"at Hannover\\ Welfengarten 1\\ 30167 Hannover\\ Germany}
\email{rohde@math.uni-hannover.de}
\begin{document}
\maketitle

\begin{abstract}
We study the generic Hodge groups $\Hg(\sX)$ of local universal deformations $\sX$ of Calabi-Yau 3-manifolds with onedimensional complex moduli, give a complete list of all possible choices for $\Hg(\sX)_{\R}$ and determine the latter real groups for known examples.
\end{abstract}

\section*{Introduction}

Let $X$ be a Calabi-Yau 3-manifold with $h^{2,1}(X) = 1$. Moreover let $f:\sX\to B$ denote the local universal deformation of $X$ and $Q$ denote the symplectic form on $H^3(X,\Q)$ given by the cup product. In the generic Hodge group $\Hg(\sX)$ information about the arithmetic of the fibers, the variation of Hodge structures and the monodromy groups of the families containing $X$ as fiber is encoded. Here we classify the possible generic Hodge groups of $\sX$, which is also a natural problem by itself.

In the case of a Calabi-Yau 3-manifold with $h^{2,1}(X) = 1$ we consider a Hodge structure on $H^3(X,\Q)$, which is a vector space of dimension 4. We have much information about the variation of Hodge structures ($VHS$) of families of Calabi-Yau 3-manifolds. For example by Bryant, Griffiths \cite{BG}, we have a classical description of the $VHS$ of such families. By using the Hodge structure on $H^3(X,\Q)$, one can construct the associated Weil- and the Griffiths intermediate Jacobians and their corresponding Hodge structures as introduced by C. Borcea \cite{Bc}. These latter Hodge structures are given by the representations $h_W$ and $h_G$  of the circle group $S^1$ on $H^3(X,\Q)$. In particular the centralizers $C(h_G(i))$ and $C(h_W(i))$ in $\Sp(H^3(X,\R),Q)$ will be helpful. By using these techniques, the theory of bounded symmetric domains \cite{Helga}, the theory of Shimura varieties \cite{De2}, \cite{Dat}, \cite{Milne}, \cite{Mo} and some intricate computations, we obtain the result:

\begin{Theorem} \label{hptres}
Let $\sX$ denote the local universal deformation of a Calabi-Yau 3-manifold $X$ with $h^{2,1}(X) = 1$. Then one of the following cases holds true:
\begin{enumerate}
 \item
$$\Hg(\sX) =\Sp(H^3(X,\Q),Q)$$
\item
$$\Hg(\sX)_{\R} = C(h_G(i))$$
\item \label{drdr} The Lie algebra of $\Hg(\sX)_{\R}$ is given by{\footnotesize
$${\rm Lie}(\Hg(\sX)_{\R})={\rm Span}_{\R}(\left(\begin{array}{cccc}
3i & 0 & 0 & 0\\
0 & i & 0 & 0 \\
0 & 0 & -i & 0\\
0 & 0 & 0 & -3i \end{array}\right),
\left(\begin{array}{cccc}
0 & 1 & 0 & 0\\
1 & 0 & x & 0 \\
0 & \bar x & 0 & 1\\
0 & 0 & 1 & 0 \end{array}\right),
\left(\begin{array}{cccc}
0 & i & 0 & 0\\
-i & 0 & ix & 0 \\
0 & -i\bar x & 0 & i\\
0 & 0 & -i & 0 \end{array}\right))$$}
for some $x \in \C$ with $|x| = \dfrac{2}{\sqrt{3}}$.
\end{enumerate}
\end{Theorem}

At present there does not exist any example of a family of Calabi-Yau 3-manifolds known to the author, which has a generic Hodge group satisfying $\eqref{drdr}$. Nevertheless we will determine the generic Hodge groups of known examples of Calabi-Yau 3-manifolds and see that there exists a Calabi-Yau like variation of Hodge structures satisfying $\eqref{drdr}$.

\tableofcontents

\section{Facts and conventions}
Here a Calabi-Yau 3-manifold $X$ is a compact K\"ahler manifold of complex dimension 3 such that
$$H^{1,0}(X) = H^{2,0}(X) = 0 \ \ \mbox{and} \ \ \omega_X \cong \sO_X.$$
We will only study Calabi-Yau 3-manifolds $X$ with $h^{2,1}(X) = 1$ here. Let $f:\sX\to B$ denote the local universal deformation of $X \cong \sX_0$, where $0\in B$.

Moreover recall the algebraic groups
$$S^1 = {\rm Spec}(\R[x,y]/x^2+y^2 - 1) \ \ \mbox{and} \ \ \BS ={\rm Spec}(\R[t,x,y]/t(x^2+y^2)-1),$$
where
$$S^1(\R)=\left\{M = \left(\begin{array}{cc}
a & b \\
-b  & a \end{array}\right) \in \SL_2(\R) \right\}\cong \{z \in \C:|z| = 1\}$$
and
$$\BS(\R) =\left\{\left(\begin{array}{cc}
a & b \\
-b  & a \end{array}\right)\in \GL_2(\R)\right\}\cong \C^*.$$
The group $\BS$ is the Deligne torus given by the Weil restriction $R_{\C/\R}(\G_m)$ and $S^1$ is a subgroup of $\BS$. Let $V$ be a real vector space. By the eigenspace decompositions of $V_{\C}$ with respect to the characters $z^p\bar z^q$ for $p,q\in \Z$ of $\BS$, the real representations $h:\BS \to \GL(V)$ correspond to the Hodge structures on $V$ (see \cite{Dat}, $1.1.1$). If there is some fixed $k$ such that all characters $z^p\bar z^q$ with non-trivial associated eigenspace satisfy $p+q = k$, one says that the Hodge structure has weight $k$. There exists an embedding $w:\G_{m,\R}\hookrightarrow\BS$ given by
$$\G_{m}(\R)\cong \{\diag(a,a) \in \GL_2(\R)\}\stackrel{\id}{\hookrightarrow} \BS(\R).$$
The Hodge structure $h$ has weight $k$, if and only if the weight homomorphism $h\circ w$ satisfies
$$r \to \diag(r^k,\ldots,r^k) \ \ (\forall \ \ r \in \R^* = \G_m(\R))$$
(see \cite{JCR}, Remark $1.1.4$). Hodge structures of some given weight $k$ are determined by the restricted representation $h|_{S^1}$.
For example the integral Hodge structure on $H^3(X,\Z)$ of weight 3 corresponds to the representation
$$h_X:S^1\to \GL(H^3(X,\R)), \ \ h_X(z)v =z^p\bar z^qv \ \  (\forall v \in H^{p,q}(X) \ \ \mbox{with} \ \ p + q = 3).$$
We also denote $h_X$ by $h$ for short. The Hodge group $\Hg(H^3(X,\Q),h)\subset\GL(H^3(X,\Q))$ is the smallest $\Q$-algebraic group $G\subset\GL(H^3(X,\Q))$ with $h(S^1) \subset G_{\R}$. Assume without loss of generality that $B$ is contractible. Thus for each $b \in B$ one has a canonical isomorphism
$$H^3(\sX_b,\Q) \cong R^3f_*(\Q)(B) \cong H^3(\sX_0,\Q)= H^3(X,\Q).$$
By using this isomorphism, a subgroup of $\GL(H^3(\sX_b,\Q))$ can be considered as a subgroup of $\GL(H^3(X,\Q))$. This allows to define an inclusion relation for the Hodge groups of the several fibers, which we use now. The generic Hodge group $\Hg(\sX)$ of $\sX$ is given by the generic Hodge group of the rational variation of Hodge structures ($VHS$) of weight 3 of $\sX$. Recall that the generic Hodge group of a $VHS$ is the maximum of the Hodge groups of all occurring Hodge structures. In an analogue way one can define the Mumford-Tate group $\MT(H^3(X,\Q),h)$ and the generic Mumford-Tate group $\MT(\sX)$ by using $h(\BS)$ instead of $h(S^1)$. One has that $\MT(H^3(\sX_b,\Q),h_b) = \MT(\sX)$ over the complement of countably many proper analytic subsets of the basis (follows from \cite{Mo}, $1.2$). Since
$$\Hg(H^3(\sX_b,\Q),h)=(\MT(H^3(\sX_b,\Q),h)\cap \SL(H^3(\sX_b,\Q)))^{0}$$
(see \cite{JCR}, Lemma $1.3.17$), one has also that $\Hg(H^3(\sX_b,\Q),h_b) = \Hg(\sX)$ over the complement of countably many proper analytic subsets of the basis.

\begin{pkt} \label{redi}
We consider only algebraic groups over fields $K$ of characteristic zero. A group $G$ over $K$ is a torus, if $G_{\bar K} \cong \G_{m,\bar K}^{\ell}$. Moreover a group $G$ is simple, if it does not contain any proper connected normal subgroup. We say that $G$ is semisimple, if its maximal connected normal solvable subgroup is trivial.

A group $G$ is reductive, if it is the almost direct product of a torus and a semisimple group. In this situation the torus can be given by the connected component of identity of the center $Z(G)$ of $G$ and the semisimple group can be given by the derived subgroup $G^{\der}$ generated by the commutators (follows from \cite{Sat}, page 9).

Let $\ad$ denote the adjoint representation. For a reductive group $G$, we have the exact sequence
$$1 \to Z(G) \to G \to G^{\ad}\to 1$$
and the adjoint group $G^{\ad}$ and $G^{\der}$ are isogenous.

We say that a semisimple group is adjoint, if its center is trivial. It is a well-known fact that connected semisimple adjoint $\R$-algebraic groups are direct products of simple subgroups.

It is a well-known fact that for a $\Q$-algebraic group $G$ the group $G_{\R}^0$ is defined over $\Q$. Moreover
$$h(S^1) \subset \Hg(\sX)_{\R}^0 \ \ \mbox{and} \ \ h(\BS) \subset \MT(\sX)_{\R}^0.$$
Thus
$$\Hg(\sX), \ \ \Hg(\sX)_{\R}, \ \ \MT(\sX) \ \ \mbox{and} \ \ \MT(\sX)_{\R}$$
are Zariski connected. Moreover Hodge groups and Mumford-Tate groups of polarized rational Hodge structures are reductive (for example see \cite{JCR}, Theorem $1.3.16$ and Corollary $1.3.20$). From this fact and the definition of reductive groups one concludes that 
$$\Hg^{\der}(\sX)_{\R}, \ \ \Hg^{\ad}(\sX)_{\R}, \ \ \MT^{\der}(\sX)_{\R} \ \ \mbox{and} \ \ \MT^{\ad}(\sX)_{\R}$$
are also Zariski connected.
\end{pkt}

By knowing the associated Lie groups of $\R$-valued points, one can determine the isomorphism classes of some algebraic groups of our interest:

\begin{lemma} \label{redzar}
Assume that $G$ and $H$ are $\R$-algebraic connected semisimple adjoint groups, where $H(\R)$ is a connected Lie group. Moreover let $h:G(\R)^+ \to H(\R)$ be an isomorphism of Lie groups. Then $G$ and $H$ are isomorphic as $\R$-algebraic groups.
\end{lemma}
\begin{proof}
From the assumptions we conclude that there is an isomorphism $dh_{\C}:\fg_{\C}\to\fh_{\C}$. Note that $\fg_{\C}$ and $\fh_{\C}$ are also semisimple as real Lie algebras and that for an arbitrary real Lie algebra $\fg'$ one can define its adjoint Lie group ${\rm Int}(\fg')$ (see \cite{Helga}, {\bf II}. $\S5$). Due to the assumption that $G$ and $H$ are semisimple adjoint, the adjoint representation yields isomorphisms
$$G(\C)^+ \cong {\rm Int}(\fg_{\C}) \ \ \mbox{and} \ \ H(\C)^+ \cong {\rm Int}(\fh_{\C}).$$
 Moreover for a real semisimple Lie algebra $\fg'$ the connected component of identity of the Lie group given by the automorphism group of $\fg'$ coincides with ${\rm Int}(\fg')$ (see \cite{Helga}, {\bf II}. Corollary $6.5$).
Thus one concludes that $G(\C)^+$ and $H(\C)^+$ are the connected components of identity of the Lie groups given by the automorphism groups of $\fg_{\C}$ and $\fh_{\C}$. Therefore one obtains a holomorphic isomorphism $h_{\C}:G(\C)^+\to H(\C)^+$. By \cite{Sat}, {\bf I}. Proposition $3.5$, the semisimple Lie groups $G(\C)^+$ and $H(\C)^+$ are the groups of $\C$-valued points of $\C$-algebraic groups and the homomorphism $h_{\C}$ is a $\C$-algebraic regular map given by some polynomials $f_1, \ldots, f_k$ over $\C$. Since $h_{\C}|_{G(\R)^+}$ coincides with $h:G(\R)^+ \to H(\R)$, one concludes that $\Im f_1, \ldots, \Im f_k$ vanish on the Zariski closure of $G(\R)^+$. The Zariski closure of $G(\R)^+$ is $G$, since we assume that $G$ is Zariski connected. Thus the isomorphism $h$ is $\R$-algebraic.
\end{proof}

\begin{pkt} \label{ctn}
Let $G$ be a connected $\R$-algebraic group and $\theta$ be an involutive automorphism of $G$. We say that $\theta$ is a Cartan involution, if the Lie subgroup
$$G^{\theta}(\R) = \{g \in G(\C)|g = \theta(\bar g)\}$$
of $G(\C)$ is compact. An $\R$-algebraic group $G$ has a Cartan involution, if and only if $G$ is reductive (see \cite{JCR}, Proposition $1.3.10$). In the case of a compact connected $\R$-algebraic group $K$ we have the Cartan involution $\id_K$ (see \cite{JCR}, Example $1.3.11$). Thus all compact connected $\R$-algebraic groups are reductive.
\end{pkt}

The Griffiths intermediate Jacobian $J_G$ resp., the Weil intermediate Jacobian $J_W$ is the torus corresponding to the weight 1 Hodge structure given by
$$F^{1}_G(H^3(X,\C)) = F^2(H^3(X,\C)) \ \ \mbox{resp.,} \  \ F^{1}_W(H^3(X,\C)) = H^{3,0}(X)\oplus H^{1,2}(X).\footnote{Note that in \cite{Bc} one has
$$F^{1}_W(H^3(X,\C)) = H^{0,3}(X)\oplus H^{2,1}(X) \ \ \mbox{instead of} \ \ F^{1}_W(H^3(X,\C)) = H^{3,0}(X)\oplus H^{1,2}(X).$$
But this is only a matter of the chosen conventions and personal preferences.}$$
Let $h_G:S^1\to \GL(H^3(X,\R))$ and $h_W:S^1\to \GL(H^3(X,\R))$ denote the corresponding representations. It is a well-known fact that weight 1 Hodge structures correspond to complex structures. We will use the complex structures
$$h_G(i) \ \ \mbox{and} \ \ h_W(i) = -h_X(i).$$
Moreover $h_W(z)$ and $h_G(z)$ commute and
$$h(z) = h_G^2(z)h_W(z).$$

Let $Q$ denote the symplectic form on $H^3(X,\Q)$ given by the cup product. For the rest of this article let us fix $v_{p,3-p}\in H^{p,3-p}(X)\setminus\{0\}$ with
$$\bar v_{p,3-p} = v_{3-p,p} \ \ \mbox{and} \ \ Q(iv_{3,0},v_{0,3}) = Q(-iv_{2,1},v_{1,2}) = 1.$$
There exist unique vectors satisfying these properties because of the well-known form of the polarization of $H^3(X,\C)$ (see \cite{Voi}, $7.1.2$) and the given Hodge numbers in our case. Thus our alternating form $Q$ on $H^3(X,\C)$ is given by
\begin{equation}\label{ququ}Q(\left(\begin{array}{c}
v_1\\
v_2\\
v_3\\
v_4\end{array}\right),\left(\begin{array}{c}
w_1\\
w_2\\
w_3\\
w_4\end{array}\right)) = (v_1,v_2,v_3,v_4)\left(\begin{array}{cccc}
0 & 0 & 0 & -i\\
0 & 0 & -i & 0\\
0 & i & 0 & 0\\
i & 0 & 0 & 0 \end{array}\right)\left(\begin{array}{c}
w_1\\
w_2\\
w_3\\
w_4\end{array}\right)\end{equation}
with respect to the basis $\{v_{3,0},v_{1,2},v_{2,1},v_{0,3}\}$.

The reader can easily check that each $M\in \GL(H^3(X,\R))$ is given by a matrix
$$M = \left(\begin{array}{cccc}
v_1 & w_1 & \bar w_4 & \bar v_4\\
v_2 & w_2 & \bar w_3 & \bar v_3\\
v_3 & w_3 & \bar w_2 & \bar v_2\\
v_4 & w_4 & \bar w_1 & \bar v_1 \end{array}\right), \ \ \mbox{where} \ \ v_1,\ldots,v_4,w_1,\ldots,w_4\in \C$$
with respect to the basis $\{v_{3,0},v_{2,1},v_{1,2},v_{0,3}\}$ by using the $\R$-vector space isomorphism given by the trace map
$$F^2(H^3(X,\C))\to H^3(X,\R), \ \ w \to w + \bar w.$$
In a similar way on can easily check that the matrices with complex entries, which will occur in this paper, are in fact real.

\begin{remark} \label{blbl}
The conjugation by elements of $h_X(S^1)(\R)$ is given by
{\scriptsize
$$ \left(\begin{array}{cccc}
\xi^3 & 0 & 0 & 0\\
0 & \xi & 0 & 0 \\
0 & 0 & \bar \xi & 0\\
0 & 0 & 0 & \bar \xi^3 \end{array}\right)
\left(\begin{array}{cccc}
a_{1,1} & a_{1,2} & a_{1,3} & a_{1,4}\\
a_{2,1} & a_{2,2} & a_{2,3} & a_{2,4}\\
a_{3,1} & a_{3,2} & a_{3,3} & a_{3,4}\\
a_{4,1} & a_{4,2} & a_{4,3} & a_{4,4}\end{array}\right)
\left(\begin{array}{cccc}
\bar \xi^3 & 0 & 0 & 0\\
0 & \bar \xi & 0 & 0 \\
0 & 0 & \xi & 0\\
0 & 0 & 0 & \xi^3 \end{array}\right)=\left(\begin{array}{cccc}
a_{1,1} & \xi^2 a_{1,2} & \xi^4 a_{1,3} & \xi^6a_{1,4}\\
\bar \xi^2 a_{2,1} & a_{2,2} & \xi^2a_{2,3} & \xi^4 a_{2,4}\\
\bar \xi^4 a_{3,1} & \bar \xi^2a_{3,2} & a_{3,3} & \bar \xi^2 a_{3,4}\\
\bar \xi^6a_{4,1} & \bar \xi^4 a_{4,2} & \bar \xi^2 a_{4,3} & a_{4,4}\end{array}\right)$$}with respect to the basis $\{v_{3,0},v_{2,1},v_{1,2},v_{0,3}\}$.
Moreover the conjugation by the elements of $h_W(S^1(\R))$ is given by:
{\scriptsize
$$ \left(\begin{array}{cccc}
\xi & 0 & 0 & 0\\
0 & \bar \xi & 0 & 0 \\
0 & 0 & \xi & 0\\
0 & 0 & 0 & \bar \xi \end{array}\right)
\left(\begin{array}{cccc}
a_{1,1} & a_{1,2} & a_{1,3} & a_{1,4}\\
a_{2,1} & a_{2,2} & a_{2,3} & a_{2,4}\\
a_{3,1} & a_{3,2} & a_{3,3} & a_{3,4}\\
a_{4,1} & a_{4,2} & a_{4,3} & a_{4,4}\end{array}\right)
\left(\begin{array}{cccc}
\bar \xi & 0 & 0 & 0\\
0 & \xi & 0 & 0 \\
0 & 0 & \bar \xi & 0\\
0 & 0 & 0 & \xi \end{array}\right)
\left(\begin{array}{cccc}
a_{1,1} & \xi^2 a_{1,2} &  a_{1,3} & \xi^2a_{1,4}\\
\bar \xi^2 a_{2,1} & a_{2,2} & \bar \xi^2a_{2,3} & a_{2,4}\\
 a_{3,1} &  \xi^2a_{3,2} & a_{3,3} &  \xi^2 a_{3,4}\\
\bar\xi^2a_{4,1} & a_{4,2} & \bar\xi^2 a_{4,3} & a_{4,4}\end{array}\right)$$}
\end{remark}

\begin{remark} \label{1.2}
The centralizer $C(h(S^1))$ of $h(S^1)$ in $\Sp(H^3(X,\R),Q)$ is given by matrices $\diag(\xi,\zeta,\bar \zeta,\bar \xi)$ with respect to the basis $\{v_{3,0},v_{2,1},v_{1,2},v_{0,3}\}$ as one concludes by the description of the conjugation by elements of $h(S^1)(\R)$ in Remark $\ref{blbl}$. Moreover by explicit computations using $\eqref{ququ}$, one concludes $|\xi| = |\zeta| = 1$. Thus $C(h(S^1))\cong S^1\times S^1$. The group of real symplectic automorphisms in $C(h(S^1))$, whose order is at most 4, is generated by $\diag(1,i,-i , 1) \ \ \mbox{and} \ \ \diag(i,1,1,-i)$. Thus $C(h(S^1))$ contains only the complex structures
\begin{equation} \label{cstr}
\pm h_W(i)=\pm\diag(i,-i,i , -i) \ \ \mbox{and} \ \ \pm h_G(i) =\pm\diag(i,i,-i,-i).
\end{equation}
Moreover $C(h(S^1))$ is generated by $h_W(S^1)$ and $h_G(S^1)$. The kernel of the natural homomorphism
$$h_W(S^1)\times h_G(S^1)\to C(h)$$
obtained from multiplication is given by $\{(1,1),(-1,-1)\}$.
\end{remark}

Let $C(h_G(i))$ and $C(h_W(i))$ denote the respective centralizers of $h_G(i)$ and $h_W(i)$ in $\Sp(H^3(X,\R),Q)$. The centralizer $C(h(i))$ of $h(i)$ in $\Sp(H^3(X,\R),Q)$ coincides with $C(h_W(i))$, since $h_W(i) = -h(i)$. Let $H$ denote the Hermitian form
$$H = iQ(\cdot,\bar \cdot).$$
Since $h(i)$ is a Hodge isometry of the real Hodge structure on $H^3(X,\R)$, one concludes from the definition of $H$ as in \cite{JCR}, Section $4.3$ and \cite{JCR3}, Lemma $3.4$:

\begin{proposition} \label{huhu}
The group $C(h_G(i))$ is given by $\diag(M,\bar M)$, where
$$M\in \U(F^2(X),H|_{F^2(X)})(\R) \cong \U(1,1)(\R)$$
and $\bar M$ acts on $\bar F^2(X)$.
\end{proposition}

In an analogue way one concludes:\footnote{It should be pretty clear to the experts that the conjugacy class of $h_W(S^1)$ in $\Sp(H^3(X,\Q),Q)$ yields the upper half plane $\fh_2$, which is also a way to conclude that $C(h_W(i))\cong \U(2)$.}

\begin{proposition} \label{uhuh}
The group $C(h_W(i))$ is given by $\diag(M,\bar M)$, where
$$M\in \U(F^2(X),H|_{H^{3,0}(X)\oplus H^{1,2}(X)})(\R) \cong \U(2)(\R)$$
and $\bar M$ acts on $H^{0,3}(X)\oplus H^{2,1}(X)$.
\end{proposition}

Thus the unitary groups $\U(1,1)$ and $\U(2)$ will be important:

\begin{remark} \label{2u}
One can describe $\U(1,1)$ and $\U(2)$ explicitly. The special unitary group $\SU(1,1)$ resp., $\SU(2)$ is given by the matrices
$$M_1 = \left(\begin{array}{cc}
\alpha & \beta \\
\bar \beta & \bar \alpha \end{array}\right) \ \ \mbox{resp.,} \ \ M_2 = \left(\begin{array}{cc}
\alpha & \beta \\
-\bar \beta & \bar \alpha \end{array}\right) \ \ \mbox{with} \ \ |\alpha|^2-|\beta|^2 = 1 \ \ \mbox{and} \ \ \alpha,\beta \in \C.$$
Moreover the reductive group $\U(1,1)$ resp., $\U(2)$ is the almost direct product of the simple group $\SU(1,1)$ resp., $\SU(2)$ and its center isomorphic to $S^1$,
where
$$\SU(1,1)\cap Z(\U(1,1)) \cong \{\pm 1\} \cong \SU(2)\cap Z(\U(2)).$$
\end{remark}

We will need an explicit description of the Lie algebra of $\SU(1,1)$:

\begin{remark} \label{lisu}
One has that
$$M_1 =\left(\begin{array}{cc}
a & b \\
\bar b & \bar a \end{array}\right)\in\SU(1,1)(\R)$$ is unipotent, if and only if
$$2 \Re(a) = tr(M_1) = 2.$$ Since each nontrivial unipotent $M_1\in\SU(1,1)(\R)$ has only one Jordan block of length 2, one computes that
$$\log M_1 =M_1 -E_2 = \left(\begin{array}{cccc}
i\Im(a) & b \\
\bar b & -i\Im(a) \end{array}\right).$$
This yields 2 linearly independent vectors of $\fsu(1,1)$ given by $\log(M_1) = M_1 -E_2$ for some unipotent $M_1$. By appending
$$\log(\diag(a,\bar a)) = \diag(iy,-iy)$$
for $|a| = 1$ and $y \in \R$, one obtains a basis of the 3-dimensional algebra $\fsu(1,1)$. Thus for each $N\in\fsu(1,1)(\R)$ there are suitable $u,v,y\in \R$ such that
$$N=
\left(\begin{array}{cccc}
iy & u+iv \\
u-iv & -iy \end{array}\right).
$$
\end{remark}

\begin{remark} \label{keinCartan}
Since the centralizer $C(h_G(i))\cong \U(1,1)$ of $h_G(i)$ is not compact, the conjugation by $h_G(i)$ does not yield a Cartan involution of $\Sp(H^3(X,\R),Q)$.
\end{remark}

\begin{lemma} \label{carty}
The conjugation by $h_W(i)$ and the conjugation by $h_X(i)$ yield the same Cartan involutions on $\Hg(\sX)_{\R}$ resp., $\Hg^{\der}(\sX)_{\R}$.
The conjugation by $\ad(h_W(i))$ yields a Cartan involution on $\Hg^{\ad}(\sX)_{\R}$.
\end{lemma}
\begin{proof}
Note that the conjugation by a complex structure
$$J\in\Sp(H^3(X,\R),Q)(\R) \ \ \mbox{with} \ \ Q(J\cdot,\cdot)> 0$$
yields a Cartan involution of $\Sp(H^3(X,\R),Q)$ (see \cite{Milne}, page 67). Since $Q(h_W(i)\cdot,\cdot)> 0$ as one can verify by using $\eqref{ququ}$ and $\eqref{cstr}$, the conjugation by $h_W(i)$ yields a Cartan involution of $\Sp(H^3(X,\R),Q)(\R)$. Due to the fact that $h_W(i) \in \Hg(\sX)_{\R}$, the conjugation by $h_W(i)$ yields a Cartan involution of the subgroup $\Hg(\sX)_{\R}\subset \Sp(H^3(X,\R),Q)$ (follows from \cite{Sat}, {\bf I}. Theorem $4.2$). Since $h_W(i)= -h_X(i)$, the conjugation by $h_X(i)$ yields the same involution.

Due to the fact that the reductive group $\Hg(\sX)_{\R}$ is the almost direct product of $Z(\Hg(\sX))^0$ and its derived group $\Hg^{\der}(\sX)_{\R}$, one concludes $h_W(i) = J_C \cdot J_D$, where $J_C\in Z(\Hg(\sX))(\C)^0$ and $J_D\in \Hg^{\der}(\sX)(\C)$. Thus
$$h_W(i) \Hg^{\der}(\sX)(\R) h_W(i)^{-1} = J_C J_D \Hg^{\der}(\sX)(\R) J_D^{-1} J_C^{-1} = J_C J_C^{-1} J_D \Hg^{\der}(\sX)(\R) J_D^{-1}$$
$$= J_D \Hg^{\der}(\sX)(\R) J_D^{-1} = \Hg^{\der}(\sX)(\R).$$
Therefore the conjugation by $h_W(i)$ yields a Cartan involution of $\Hg^{\der}(\sX)_{\R}$. This Cartan involution corresponds clearly to a Cartan involution on $\Hg^{\ad}(\sX)_{\R}$ given by the conjugation by $\ad(h_W(i))$.
\end{proof}

Let $K$ be a maximal compact subgroup of $\Hg(\sX)_{\R}$. Since all maximal compact subgroups of a reductive group are conjugate, we assume without loss of generality that $K$ is the subgroup fixed by the Cartan involution obtained from conjugation by $h_W(i)$. Let $C((\ad\circ h)(i))$ denote the centralizer of $(\ad\circ h)(i)$ in $\Hg^{\ad}(\sX)$.

\begin{lemma} \label{centis}
$$C((\ad\circ h)(i)) = \ad(K)=\ad(K\cap\Hg^{\der}(\sX)_{\R})$$
\end{lemma}
\begin{proof}
One has clearly
$$C((\ad\circ h)(i))\supseteq\ad(K)\supseteq\ad(K\cap\Hg^{\der}(\sX)_{\R}).$$
Thus it remains to prove
$$C((\ad\circ h)(i)) \subseteq\ad(K\cap\Hg^{\der}(\sX)_{\R}).$$
Since $\Hg^{\ad}(\sX)_{\R}$ and $\Hg^{\der}(\sX)_{\R}$ are isogenous, we have a correspondence between their maximal compact subgroups. The maximal compact subgroups $K_G$ of real algebraic reductive groups $G$ are the subgroups of $G$ satisfying
$$K_G = \{g \in G \ \ | \ \ \theta(g) = g\}$$
for some Cartan involution $\theta$ (follows from \cite{Sat}, {\bf I}. Corollary $4.3$ and Corollary $4.5$). Recall that the conjugation by $h(i)$ yields a Cartan involution on $\Hg^{\der}(\sX)_{\R}$ and the conjugation by $(\ad\circ h)(i)$ yields a Cartan involution on $\Hg^{\ad}(\sX)_{\R}$. Thus one concludes that the centralizer of $(\ad\circ h)(i)$ is given by $\ad(K\cap\Hg^{\der}(\sX)_{\R})$.
\end{proof}

\section{Computation of the adjoint Hodge group}

In this section we prove the following theorem:

\begin{Theorem} \label{tzu}
The group $\Hg^{\ad}(\sX)_{\R}$ is isomorphic to $\PU(1,1)$ or $\Sp^{\ad}_{\R}(4)$.
\end{Theorem}

 For the proof of Theorem $\ref{tzu}$ we need to understand $K$ first:

\begin{lemma} \label{kaka}
The group $K^0$ is a torus or  $K= C(h_W(i))$.
\end{lemma}
\begin{proof}
Since $K^0$ is compact, $K^0$ is reductive (see $\ref{ctn}$). One has without loss of generality
$$K \subseteq C(h_W(i)) \cong \U(2).$$
If $K^0$ is a torus, we are done. Otherwise $K^0$ has a nontrivial semisimple subgroup
$$K^{\der}\subseteq C^{\der}(h_W(i)) \cong \SU(2)$$
(see $\ref{redi}$). Since $\SU(2)$ does not contain any simple proper subgroup, $K^{\der}= C^{\der}(h_W(i))$. From the facts that $h(S^1)$ is not contained in $C^{\der}(h_W(i))$, but contained in $\Hg(\sX)_{\R}$ and commutes with $h_W(i)= h(-i)$, we conclude $K= C(h_W(i))$ in this case.
\end{proof}

\begin{lemma} \label{komzen}
The centralizer of $C^{\der}(h_W(i))$ in $\Sp(H^3(X,\R),Q)$ is given by the center $Z(C(h_W(i)))$ of $C(h_W(i))$.
\end{lemma}
\begin{proof}
Recall the description of $C^{\der}(h_W(i))\cong \SU(2)$ in Proposition $\ref{uhuh}$ and the description of $\SU(2)$ in Remark $\ref{2u}$. Thus $N \in C^{\der}(h_W(i))(\R)$ is given by
$$N = \left(\begin{array}{cccc}
a & b & 0 & 0\\
-\bar b  &  \bar a & 0 & 0\\
0 & 0 & a & -b\\
0 & 0 & \bar b  & \bar a \end{array}\right) \ \ \mbox{with} \ \ |a|^2+|b|^2 = 1$$
with respect to the basis $\{v_{3,0},v_{1,2},v_{2,1},v_{0,3}\}$. Now let 
$$M = \left(\begin{array}{cc}
A & B\\
C & D \end{array}\right)\in \Sp(H^3(X,\R),Q)(\R)$$
commute with each $N \in C^{\der}(h_W(i))(\R)$ for some suitable $A,B,C,D \in \GL_2(\C)$. Thus $M$ commutes with $\diag(i,-i,i,-i)$ and one computes that $A,B,C,D$ are diagonal matrices. Moreover one has that $M$ has to commute with
$$N = \left(\begin{array}{cccc}
0 & 1 & 0 & 0\\
-1  &  0 & 0 & 0\\
0 & 0 & 0 & -1\\
0 & 0 & 1  & 0 \end{array}\right).$$
From this fact and the assumptions that  $M$ is a real matrix and commutes with each element of $C^{\der}(h_W(i))(\R)$, one concludes
$$M =\left(\begin{array}{cccc}
z & 0 & \bar y & 0\\
0 & z & 0 & -\bar y\\
-y & 0 & \bar z & 0\\
0 & y & 0 & \bar z \end{array}\right).$$
Moreover one computes that
$$M^tQM=\left(\begin{array}{cccc}
z & 0 & - y & 0\\
0 & z & 0 &  y\\
\bar y & 0 & \bar z & 0\\
0 & -\bar y & 0 & \bar z \end{array}\right)
\left(\begin{array}{cccc}
0 & 0 & 0 & -i\\
0 & 0 & -i & 0\\
0 & i & 0 & 0\\
i & 0 & 0 & 0 \end{array}\right)
\left(\begin{array}{cccc}
z & 0 & \bar y & 0\\
0 & z & 0 & -\bar y\\
-y & 0 & \bar z & 0\\
0 & y & 0 & \bar z \end{array}\right)$$
$$=
\left(\begin{array}{cccc}
0 & -2iyz & 0 & i|y|^2-i|z|^2\\
2iyz & 0 & i|y|^2-i|z|^2 & 0\\
0 & -i|y|^2+i|z|^2 & 0 & -2i\bar y\bar z\\
-i|y|^2+i|z|^2 & 0 & 2i\bar y\bar z & 0 \end{array}\right).$$
Hence $M\in \Sp(H^3(X,\R),Q)$, only if $y= 0$ and $|z|=1$. Thus $M \in Z(C(h_W(i)))$.
\end{proof}

\begin{lemma} \label{ccp}
$\Hg(\sX)_{\R}$ cannot be compact.
\end{lemma}
\begin{proof}
Assume that $\Hg(\sX)_{\R}$ would be compact. Thus one concludes that $\Hg(\sX)_{\R} = K$ is a torus or $\Hg(\sX)_{\R} = C(h_W(i))$ (see Lemma $\ref{kaka}$). In the first case one concludes $\Hg(\sX)_{\R} \subseteq C(h(S^1))$, which contains only 4 complex structures (see Remark $\ref{1.2}$). In the second case the Cartan involution obtained from conjugation by $h_{\sX_b}(i)\in C(h_W(i))$ fixes each element of the compact group $\Hg(\sX)(\R) = C(h_W(i))(\R)$ for each $b \in B$. Hence each $h_{\sX_b}(i)$ has to be contained in the center of $C(h_W(i))$. Note that $Z(C(h_W(i)))$ has only the two complex structures $\pm h_W(i)$. Thus in any case $h(i) = h_{\sX_b}(i)$ for each $b \in B$, since the $VHS$ is continuous and for each $b \in B$ one obtains
$$H^{3,0}(\sX_b)\subset {\rm Eig}(h_{\sX_b}(i), -i)= {\rm Eig}(h_X(i), -i) = {\rm Span}(v_{3,0}, v_{1,2}).$$
But this contradicts the fact that $\omega(0)$ and $\nabla_{\frac{\partial}{\partial b}}\omega(0)$ generate $F^2(X)$, where $\omega$ denotes a generic section of the $F^3$-bundle in the $VHS$ (see \cite{BG}).
\end{proof}

Now we change for a moment to the language of semisimple adjoint Lie groups. Connected semisimple adjoint Lie groups are direct products of their normal simple subgroups (see \cite{JCR}, Lemma $1.3.8$). The group $\Hg^{\ad}(\sX)(\R)^+$ is an example of a connected semisimple adjoint Lie group.

\begin{proposition} \label{zk}
There does not exist any nontrivial direct factor $F$ of $\Hg^{\ad}(\sX)(\R)^+$ such that $$Z(K)(\R)^+\subset \ker (pr_F\circ \ad).$$
\end{proposition}
\begin{proof}
Assume that $F$ is a direct factor of $\Hg^{\ad}(\sX)(\R)^+$ with
$$Z(K)(\R)^+\subset \ker ( pr_F\circ \ad).$$
We show that $F$ is trivial. Since $\Hg(\sX)_{\R}$ cannot be compact (see Lemma $\ref{ccp}$), the maximal compact subgroup $K$ associated to the Cartan involution obtained from conjugation by $h(i)$ is a proper subgroup. Thus $h(i)$ is not contained in the center of $\Hg(\sX)_{\R}$. Recall that $K^0$ is a torus or $K = C(h(i))$. Since $h(S^1)(\R)$ is connected, $h(i) \in Z(K)(\R)^+$ in both cases. Thus from our assumption we conclude that $F$ is contained in the maximal compact subgroup associated to the Cartan involution obtained from conjugation by $(\ad\circ h)(i)$. Consider the projection map $pr_F: \Hg^{\ad}(\sX)(\R)^+ \to F$. Since $$(\ad\circ h)(i)\in G:=\ker (pr_F)\subset \Hg^{\ad}(\sX)(\R)^+,$$
one concludes that $G$ is non-trivial semisimple adjoint. Note that
$$\ker (pr_G)= F \ \ \mbox{and} \ \ \Hg^{\ad}(\sX)(\R)^+ = F\times G,$$
since connected semisimple adjoint Lie groups are direct products of their normal simple subgroups (see \cite{JCR}, Lemma $1.3.8$). Let
$$F' = \ker(pr_G \circ \ad|_{\Hg^{\der}(\sX)(\R)})^+ \ \ and \ \ G' =\ker(pr_F \circ \ad|_{\Hg^{\der}(\sX)(\R)})^+.$$
Since $\Hg^{\ad}(\sX)_{\R}$ and $\Hg^{\der}(\sX)_{\R}$ are isogenous, one concludes that
$F'$ and $G'$ commute. Since $F'$ is a semisimple group with elements fixed by the Cartan involution obtained from conjugation by $h(i)$ and $C^{\der}(h(i))(\R)\cong \SU(2)(\R)$ contains no semisimple proper subgroup, one concludes $$F' = C^{\der}(h(i))(\R) \ \ \mbox{or} \ \ F' = \{e\}.$$
Only the torus $Z(C(h(i)))$ commutes with $C^{\der}(h(i))$ (see Lemma $\ref{komzen}$). Thus from the fact that $G'$ is nontrivial semisimple and commutes with $F'$, we conclude $F' = \{e\}$. Thus $F$ is trivial.
\end{proof}

The connected semisimple adjoint Lie group $\Hg^{\ad}(\sX)(\R)^+$ is a direct product of connected simple adjoint subgroups. Let $F$ be one of these nontrivial direct factors. The maximal compact subgroup of $\Hg^{\ad}(\sX)(\R)^+$ is given by
$$\ad(K(\R))\cap\Hg^{\ad}(\sX)(\R)^+$$
(follows from Lemma $\ref{centis}$).
Thus for the maximal compact subgroup $K_F$ of $F$ one concludes that $K_F^+ = (pr_F\circ \ad)(K(\R)^+)$. Due to the fact that $Z(K)(\R)^+$ is not contained in $\ker(pr_F\circ \ad)$ and not discrete as one concludes from Lemma $\ref{kaka}$, the maximal compact subgroup $K_F$ has a nondiscrete center. Since $F$ has a trivial center, $K_F \neq F$ and one concludes:

\begin{corollary} \label{zk2}
The connected adjoint Lie group $\Hg^{\ad}(\sX)(\R)^+$ is a direct product of noncompact simple adjoint subgroups, whose maximal compact subgroups have nondiscrete centers.
\end{corollary}

Note that each Hermitian symmetric domain is a direct product of irreducible Hermitian symmetric domains (for the definition and more details about Hermitian symmetric domains see \cite{Helga}). If $G$ is a connected simple adjoint noncompact Lie group and $K_G$ is a maximal compact subgroup of $G$ with nondiscrete center, the quotient $G/K_G$ has the structure of a uniquely determined irreducible Hermitian symmetric domain (\cite{Helga}, {\bf XIII}. Theorem $6.1$,). Hence one concludes from Corollary $\ref{zk2}$:

\begin{proposition} \label{juhuu}
The quotient
$$D =\Hg^{\ad}(\sX)(\R)^+/\ad(K(\R))\cap\Hg^{\ad}(\sX)(\R)^+$$
has the structure of an Hermitian symmetric domain.
\end{proposition}

Since $\Hg(\sX)_{\R} \subset \Sp(H^3(X,\R),Q)$, the associated Hermitian symmetric domain of $\Sp(H^3(X,\Q),Q)(\R)$ is $\fh_2$ and $\dim_{\C}\fh_2 = 3$, the Hermitian symmetric domain $D$ has dimension 1, 2 or 3. By using these conditions, we obtain some candidates for $\Hg^{\ad}(\sX)(\R)^+$. Since these candidates are the Lie groups of real valued points of $\R$-algebraic semisimple adjoint groups, we obtain not only connected Lie groups, but $\R$-algebraic groups in our cases by using Lemma  $\ref{redzar}$. Moreover we will exclude all of these candidates except of the candidates stated in Theorem $\ref{tzu}$.

\begin{lemma}
If $D$ has dimension one, we obtain
$$\Hg^{\ad}(\sX)_{\R}\cong \PU(1,1).$$
\end{lemma}
\begin{proof}
Assume that $D$ has dimension one. By consulting the list of irreducible Hermitian symmetric domains (\cite{Helga}, {\bf X}, Table {\bf V}), one concludes $D = \B_1$. Thus from the fact that there are no direct compact factors (see Corollary $\ref{zk2}$) one concludes
$$\Hg^{\ad}(\sX)_{\R}\cong \PU(1,1).$$
\end{proof}

\begin{lemma}
If $D$ has dimension two, we obtain
$$\Hg^{\ad}(\sX)_{\R}\cong\PU(1,2), \ \ \mbox{or} \ \ \Hg^{\ad}(\sX)_{\R}\cong\PU(1,1)\times \PU(1,1).$$
\end{lemma}
\begin{proof}
By consulting the list of irreducible Hermitian symmetric domains (\cite{Helga}, {\bf X}, Table {\bf V}), the only possible Hermitian symmetric domains of dimension two are up to isomorphisms given by $\B_1 \times \B_1$ and $\B_2$. Thus we obtain the stated result.
\end{proof}

\begin{lemma}
One obtains $\Hg(\sX) = \Sp(H^3(X,\Q),Q)$, if $D$ has the dimension 3.
\end{lemma}
\begin{proof}
We show that $\fh_2$ contains no bounded symmetric domain of dimension 3 except of itself. In order to do this we check the list of Hermitian Symmetric Domains (compare \cite{Helga}, {\bf X}, Table {\bf V}). The domain $D$ cannot be the direct product of 3 copies of $\B_1$, since in this case the centralizer of $(\ad\circ h_X)(i)$ would be a torus of dimension 3. But the centralizer of $h_X(i)$ is isomorphic to $\U(2)$, which contains a maximal torus of dimension 2. Since each point $p \in\B_1\times \B_2$ has a centralizer $S^1\times \U(2)$ of dimension 5 and $C(h(i)) \cong\U(2)$ has dimension 4, one concludes that $D$ cannot be isomorphic to $\B_1\times \B_2$. In the case of $\B_3$ the stabilizer is $\U(3)$ and hence it is to large. The same holds true in the case of $\SO^*(6)/\U(3)$. Moreover the associated bounded symmetric domain of $\SO(2,3)^+(\R)$ is isomorphic to $\fh_2$. Thus we obtain the stated result.
\end{proof}

By the previous lemmas, the following adjoint semisimple groups are possible candidates for $\Hg^{\ad}(\sX)_{\R}$:
$$\PU(1,1), \ \ \PU(1,1)\times \PU(1,1), \ \ \PU(1,2), \ \ \Sp^{\ad}_{\R}(4)$$
Now we exclude $\PU(1,2)$ and $\PU(1,1)\times \PU(1,1)$.

\begin{proposition}
The group $\Hg^{\ad}(\sX)_{\R}$ cannot be isomorphic to $\PU(1,2)$.
\end{proposition}
\begin{proof}
Assume that $\Hg^{\ad}(\sX)_{\R}$ would be isomorphic to $\PU(1,2)$. In this case the centralizer $C((\ad \circ h)(i))\subset \Hg^{\ad}(\sX)_{\R}$ of the complex structure $(\ad \circ h)(i)$ is isomorphic to $\U(2)$. Hence $C((\ad \circ h)(i))$ has dimension 4. One has that $C((\ad \circ h)(i))$ is isogenous to $C(h(i))\cap \Hg^{\der}(\sX)_{\R}$. Since $C(h(i))$ has already dimension 4 and $h(S^1) \subset C(h(i))$, one concludes
$$C(h(i))\subset \Hg^{\der}(\sX)_{\R} \ \ \mbox{and} \ \ \Hg^{\der}(\sX)_{\R}= \Hg(\sX)_{\R}.$$
Note that
$$C^{\der}(h(i)) \cong \SU(2).$$
Moreover $\ad$ yields a homomorphism
$$g:=\ad|_{C^{\der}(h(i))}: C^{\der}(h(i)) \to C(\ad \circ h(i)),$$
whose kernel consists of $\{\pm \id\}$. Since
$$C^{\der}(h(i))/\{\pm \id\} \cong \PU(2)$$
is semisimple, one has
$$(g(C^{\der}(h(i))))^{\der} = g(C^{\der}(h(i))).$$
Hence
$$g(C^{\der}(h(i)))\subset C^{\der}(\ad \circ h(i))\cong \SU(2).$$
Recall that
$$\SU(2)(\R) = \{M(\alpha,\beta)=\left(\begin{array}{cc}
\alpha & \beta \\
-\bar \beta  &   \bar \alpha \end{array}\right): |\alpha|^2+|\beta|^2 = 1\}.$$
Each matrix $M(\alpha,\beta) \in \SU(2)(\R)$ with $\alpha \in i\R$ has the characteristic polynomial
$$x^2+1=(x-i)(x+i),$$
which implies that $M(\alpha,\beta)$ is a complex structure. Therefore $C^{\der}(h(i))(\R)\cong\SU(2)(\R)$ contains infinitely many complex structures. Since $\ker(g)=\{\pm \id\}$, all these complex structures are mapped to infinitely many elements of order 2 in $C^{\der}(\ad \circ h(i))$. Since each $2\times2$ matrix $M$ of order 2 has a minimal polynomial dividing the polynomial $x^2-1$, the matrix $M$ is either given by $\diag(-1,-1)$ or one has an eigenspace  with respect to 1 and one eigenspace with respect to $-1$. In the second case $\det(M) = -1$. Thus $\diag(-1,-1)$ is the only element of order 2 in $\SU(2)(\R)$. On the other hand there are infinitely many complex structures in $C^{\der}(h(i))(\R)$, which are mapped by $g$ to infinitely many elements of order 2 in $C((\ad\circ h)(i))(\R) \cong \SU(2)(\R)$. Thus we have a contradiction.
\end{proof}

Let $H$ denote the centralizer of $h_G(i)h_W(i)$ in $\Sp(H^3(X,\R),Q)$. Note that
$$h_G(i)h_W(i) = \diag(-1,-1,1,1)$$
with respect to the basis $\{v_{3,0}, v_{0,3}, v_{2,1}, v_{1,2}\}$.
Thus $H(\R)$ is given by the matrices
\begin{equation} \label{haha} M_1 = \left(\begin{array}{cccc}
a & b & 0 & 0 \\
\bar b  &  \bar a & 0 & 0 \\
0 & 0 & c & d \\
0 & 0 & \bar d & \bar c \end{array}\right) \ \ \mbox{with} \ \ \left(\begin{array}{cc}
a & b \\
\bar b  &  \bar a\end{array}\right),
\left(\begin{array}{cc}
c & d \\
\bar d  &  \bar c\end{array}\right) \in \SU(1,1)(\R)
\end{equation}
with respect to the basis $\{v_{3,0}, v_{0,3}, v_{2,1}, v_{1,2}\}$.
One can easily verify this fact by explicit computations using the description of the symplectic form $Q$ in $\eqref{ququ}$. The group $H$ will play an important role due to the following lemma:

\begin{lemma} \label{keinhaar}
The group $\Hg(\sX)_{\R}$ cannot be a subgroup of $H$.
\end{lemma}
\begin{proof}
Assume that $\Hg(\sX)_{\R}$ would be a subgroup of $H$. Since for each $b \in B$ the conjugation by $h_W(i)_b$ yields a Cartan involution of $\Sp(H^3(X,\R),Q)$, which can be restricted to an involution of $H$ in this case,  the conjugation by $h_W(i)_b$ yields a Cartan involution of $H$ (compare \cite{Sat}, {\bf I}. Theorem $4.2$). Due to the fact $H\cong \SU(1,1)\times\SU(1,1)$, the corresponding maximal compact subgroup is a torus of dimension 2 containing $h_b(S^1)$. By Remark $\ref{1.2}$, the centralizer $C(h_b(S^1))$ is already a torus of dimension 2. Hence
$$h_G(i)_b \in C(h_b(S^1)) \subset H.$$
Thus from the description of $H$ in $\eqref{haha}$ and the fact that $h_G(i)_b,h_W(i)_b\in H$ are real complex structures, one concludes that
$${\rm Eig}(h_G(i)_b,i) = {\rm Span}(v_1,v_3), \ \ {\rm Eig}(h_W(i)_b,i) = {\rm Span}(v_2,v_4)$$
with
\begin{equation} \label{vis}
v_1,v_2 \in {\rm Span}(v_{3,0}, v_{0,3}), \ \ v_3,v_4 \in {\rm Span}(v_{2,1}, v_{1,2}).
\end{equation}
For each $b\in B$ one has the onedimensional vector space
$$H^{3,0}(\sX_b) = {\rm Eig}(h_G(i)_b,i) \cap {\rm Eig}(h_W(i)_b,i).$$
Hence $\{v_1, \ldots, v_4\}$ is not linearly independent and one concludes from the description of $H$ in $\eqref{vis}$ that $H^{3,0}(\sX_b)$ is either contained in ${\rm Span}(v_{3,0}, v_{0,3})$ or contained in ${\rm Span}(v_{2,1}, v_{1,2})$.\footnote{This is only an exercise in linear algebra.} Since the period map is continuous, one has for each $b \in B$
$$H^{3,0}(\sX_b)\subset {\rm Span}(v_{3,0}, v_{0,3}).$$
This contradicts the fact that $\omega(0)$ and $\nabla_{\frac{\partial}{\partial b}}\omega(0)$ generate $F^2(X)$, where $\omega$ denotes a generic section of the $F^3$-bundle in the $VHS$ (see \cite{BG}). Thus $\Hg(\sX)_{\R}$ cannot be a subgroup of $H$.
\end{proof}

\begin{proposition} \label{2b}
One cannot have
$$\Hg^{\ad}(\sX)_{\R}\cong\PU(1,1)\times \PU(1,1).$$
\end{proposition}
\begin{proof}
Assume that $\Hg^{\ad}(\sX)_{\R}\cong\PU(1,1)\times \PU(1,1)$. Without loss of generality the only possible Cartan involution of $\PU(1,1)\times \PU(1,1)$ is given by the conjugation by
$$([\diag(i,-i)],[\diag(i,-i)]) \in \PU(1,1)\times \PU(1,1).$$
Moreover in $\Hg^{\ad}(\sX)_{\R}\cong\PU(1,1)\times \PU(1,1)$ the maximal compact subgroup of elements fixed by the Cartan involution is given by a torus of dimension 2. Thus there is a torus $T\subset \Hg^{\der}(\sX)_{\R}$ of dimension two, whose elements are fixed by the Cartan involution. Assume without loss of generality that the Cartan involution of $\Hg^{\der}(\sX)_{\R}$ is obtained from conjugation by $h(i)$.
Thus $T$ is a maximal torus of $C(h(i))\cong \U(2)$, since $T$ has dimension 2. Therefore the center of $\Hg(\sX)_{\R}$ is discrete and one concludes from $\ref{redi}$ that
$$\Hg^{\der}(\sX)_{\R}= \Hg(\sX)_{\R}.$$
From the fact that each element of $h(S^1)$ commutes with $h(i)$, one concludes $h(S^1)\subset T$. Since $T$ is a torus of dimension 2 containing $h(S^1)$, one concludes from Remark $\ref{1.2}$ that $T = C(h(S^1))$. Thus $h_G(i) \in T$ and $h_G(S^1) \subset T$. Note that $h_G(i)$ cannot be contained in the center of $\Hg(\sX)_{\R}$, since $h_G(i)\in Z(\Hg(\sX)_{\R})$ would imply that $h_G(S^1)\subset Z(\Hg(\sX)_{\R})$ as one can easily conclude from the fact that
$$h_G(S^1)(\R) = \{a\cdot\id+b\cdot h_G(i) \ \ | \ \ a^2+b^2 = 1\}.$$
This contradicts our conclusion that $Z(\Hg(\sX)_{\R})$ is discrete. Since $h_G(i)$ has order 4 and $$h_G(i)^2 = -\id\in \ker(\ad),$$ one concludes that $\ad(h_G(i))$ yields an element of order two in $\ad(T)$. Note that $\ad(T)$ has only the three elements 
$$([\diag(i,-i)],[\diag(1,1)]), \ \ ([\diag(i,-i)],[\diag(i,-i)]) \ \ \mbox{and} \ \ ([\diag(1,1)],[\diag(i,-i)])$$
of order 2. Thus we have two cases: 
In the first case $\ad(h_G(i))$ is without loss of generality given by
$$([\diag(i,-i)],[\diag(1,1)]).$$
Let $pr_i$ ($i=1,2$) denote the projection of $\Hg^{\ad}(\sX)_{\R}\cong\PU(1,1)\times \PU(1,1)$ to the respective copy of $\PU(1,1)$. One has that $\Hg(\sX)_{\R}$ contains $\ker(pr_1\circ \ad)^{0}$ and $\ker(pr_2\circ \ad)^0$. Since the groups $\Hg^{\ad}(\sX)_{\R}$ and $\Hg^{\der}(\sX)_{\R}=\Hg(\sX)_{\R}$ are isogenous, $\ker(pr_1\circ \ad)^0$  and $\ker(pr_2\circ \ad)^0$ are also isogenous to $\PU(1,1)$ and commute also. Moreover since $\Hg^{\ad}(\sX)_{\R}$ and $\Hg(\sX)_{\R}$ are isogenous, $(\Hg(\sX)_{\R}\cap C(h_G(i)))^0$ is also isogenous to $C((\ad\circ h_G)(i))$. Since $\ker(pr_1)$ commutes with $\ad(h_G(i))$, one concludes that $\ker(pr_1\circ \ad)^0$ is a nontrivial simple subgroup of  $C(h_G(i))$. Since the only nontrivial simple subgroup of $C(h_G(i))$ is $C^{\der}(h_G(i))$, one gets
$$\ker(pr_1\circ \ad)^0 = C^{\der}(h_G(i)).$$
By analogue arguments, one concludes
$$\ker(pr_2\circ \ad)^0\subset H:= C(h_G(i)h_W(i)).$$
We obtain the desired contradiction by showing that $\ker(pr_1\circ \ad)^0$  and $\ker(pr_2\circ \ad)^0$ cannot commute here.
One has that $C^{\der}(h_G(i))(\R)$ is given by matrices of the form
$$M_2 = \left(\begin{array}{cccc}
\alpha &  0 & \beta & 0 \\
0  &  \bar \alpha & 0 & \bar \beta \\
\bar \beta & 0 & \bar \alpha & 0 \\
0 & \beta & 0 & \alpha \end{array}\right) \ \ \mbox{with} \ \ |\alpha|^2-|\beta|^2 = 1$$
with respect to the basis
$$\{v_{3,0}, v_{0,3}, v_{2,1}, v_{1,2}\}$$
as the reader can easily verify by the description of $C(h_G(i))(\R)\cong \U(1,1)$ in Proposition $\ref{huhu}$ and the description of $\SU(1,1)$ in Remark $\ref{2u}$. Moreover by explicit computations using $\eqref{haha}$, one checks that in $H(\R)$ only the diagonal matrices of the kind $\diag(\xi,\bar \xi,\xi,\bar \xi)$ commute with each element of $C^{\der}(h_G(i))(\R)$. This contradicts our previous conclusion that $H$ contains a subgroup isogenous to $\PU(1,1)$, which commutes with $C^{\der}(h_G(i))(\R)$. Hence the first case cannot hold true.

In the second case $\ad(h_G(i))$ is given by
$$([\diag(i,-i)],[\diag(i,-i)]) \in \PU(1,1)\times \PU(1,1).$$
This implies that $\Hg^{\der}(\sX)= \Hg(\sX)_{\R}$ is contained in the subgroup of $\Sp(H^3(X,\R),Q)$ on which both involutions obtained from conjugation by $h_W(i)$ and $h_G( i)$ coincide. One has that $$h_W(i)=\diag(i,-i,-i,i) \ \ \mbox{and} \ \ h_G(i)=\diag(i,-i,i,-i)$$
with respect to the basis
$$\{v_{3,0}, v_{0,3}, v_{2,1}, v_{1,2}\}.$$
Thus $H$ is the subgroup of $\Sp(H^3(X,\R))$ on which both involutions obtained from conjugation by $h_W(i)$ and $h_G(i)$ coincide as one can easily compute by using the description of $H$ in $\eqref{haha}$. But by Lemma $\ref{keinhaar}$, the group $H$ cannot contain  $\Hg(\sX)_{\R}$. Thus the second case cannot occur.
\end{proof}

\section{The case of a onedimensional period domain}

In this section we will assume that the period domain $D$ has dimension 1 unless stated otherwise. In the previous section we saw that $\Hg^{\ad}(\sX)_{\R}\cong \PU(1,1)$, if $D = 1$. Since
$$\Hg(\sX) = (\SL(H^3(X,\Q)) \cap \MT(\sX))^0$$
(follows from \cite{JCR}, Lemma $1.3.17$), one concludes
$$\Hg^{\ad}(\sX) = \MT^{\ad}(\sX).$$

Recall the definition of Shimura data:

\begin{definition}
Let $G$ be a reductive $\Q$-algebraic group and $h:\BS\to G_{\R}$ be a homomorphism. Then the pair $(G,h)$ is a Shimura datum, if:
\begin{enumerate}
 \item  The group $G^{\ad}$ has no nontrivial direct compact factor over $\Q$.
 \item The conjugation by $h(i)$ is a Cartan involution.
\item The representation $\ad \circ h$ of $\BS$ on ${\rm Lie}(G_{\R})$ is a Hodge structure of type $$(1,-1),(0,0),(-1,1).$$
\end{enumerate}
\end{definition}

We will show that the pair $(\MT(\sX),h_X)$ is a Shimura datum. Moreover we will determine the center of $\Hg(\sX)_{\R}$ and $\Hg(\sX)_{\R}$ in the case of a nondiscrete center. In addition we describe the monodromy in the latter case and give some examples.

\begin{proposition} \label{cent}
The center of $\Hg(\sX)(\R)$ is given by diagonal matrices $\diag(\xi,\xi,\bar\xi,\bar\xi)$ for $|\xi| = 1$ with respect to the basis $\{v_{3,0},v_{2,1},v_{1,2},v_{0,3}\}$.
\end{proposition}
\begin{proof}
Each element $Z$ in the center of $\Hg(\sX)(\R)$ commutes in particular with $h_X(S^1)(\R)$. This holds only true, if $Z$ is a diagonal matrix with respect to $\{v_{3,0},v_{2,1},v_{1,2},v_{0,3}\}$ as the conjugation by elements of $h(S^1)(\R)$ in Remark $\ref{blbl}$ shows. The subgroup of the matrices in $\Sp(H^3(X,\R),Q)$, which are diagonal with respect to $\{v_{3,0},v_{2,1},v_{1,2},v_{0,3}\}$, is contained in $C(h_W(i))\cong \U(2)$ and therefore compact. By Lemma $\ref{ccp}$, the group $\Hg(\sX)_{\R}$ cannot be compact. Thus $\Hg(\sX)_{\R}$ contains elements, which are not given by diagonal matrices with respect to $\{v_{3,0},v_{2,1},v_{1,2},v_{0,3}\}$. Since $Z$ has to be real and to commute with the matrices in $\Hg(\sX)(\R)$, which are not diagonal, one concludes that
$$Z =\pm\diag(\xi,1,1,\bar\xi), \ \ Z =\pm\diag(1,\xi,\bar\xi,1),$$
$$Z =\diag(\xi,\xi,\bar\xi,\bar\xi) \ \ \mbox{or} \ \ Z =\diag(\xi,\bar\xi,\xi,\bar\xi)$$
with respect to the basis $\{v_{3,0},v_{2,1},v_{1,2},v_{0,3}\}$. Moreover one has $|\xi| = 1$, since $Z^tQZ=Q$. For $Z =\pm\diag(\xi,1,1,\bar\xi)$ with $\xi \neq \pm 1$ the centralizer $C(Z)$ of $Z$ in $\Sp(H^3(X,\R),Q)$ is given by the group of matrices
\begin{equation}\label{bizarre}
M =\left(\begin{array}{cccc}
\zeta & 0 & 0 & 0\\
0 & \alpha & \beta & 0 \\
0 & \bar \beta & \bar \alpha & 0\\
0 & 0 & 0 & \bar \zeta \end{array}\right) \ \ \mbox{with} \ \ |\zeta| = 1 \ \ \mbox{and} \ \ \left(\begin{array}{cc}
\alpha & \beta \\
 \bar \beta & \bar \alpha \end{array}\right)\in \SU(1,1)
 \end{equation}
as one concludes by computations using $\eqref{ququ}$. Thus one concludes that $C(Z)\subset H$ from the description of $H$ in $\eqref{haha}$.

Moreover for $Z =\pm\diag(1,\xi,\bar\xi,1)$ with $\xi \neq \pm 1$ the centralizer $C(Z)$ is given by
$$
M =\left(\begin{array}{cccc}
\alpha & 0 & 0 & \beta\\
0 & \zeta & 0 & 0 \\
0 & 0 & \bar \zeta & 0\\
\bar \beta & 0 & 0 & \bar \alpha \end{array}\right) \ \ \mbox{with} \ \ |\zeta| = 1 \ \ \mbox{and} \ \ \left(\begin{array}{cc}
\alpha & \beta \\
 \bar \beta & \bar \alpha \end{array}\right)\in \SU(1,1),
$$
which is also a subgroup of $H$ as one concludes from analogue arguments. By Lemma $\ref{keinhaar}$, the group $\Hg(\sX)_{\R}$ cannot be a subgroup of $H$. Since the matrices of the form
$$\pm\diag(\xi,1,1,\bar\xi), \ \ \pm\diag(1,\xi,\bar\xi,1) \ \ \mbox{with} \ \ \xi \neq \pm 1$$
have centralizers contained in $H$, these matrices are not contained in the center of $\Hg(\sX)_{\R}$.

One can also not have that
$$Z=\pm\diag(1,-1,-1,1) \in Z(\Hg(\sX)_{\R}),$$
too, since in this case the centralizer of $Z$ in $\Sp(H^3(X,\R),Q)$ is $H$.

Hence one has
$$Z =\diag(\xi,\xi,\bar\xi,\bar\xi) \ \ \mbox{or} \ \ Z =\diag(\xi,\bar\xi,\xi,\bar\xi).$$
The matrix $\diag(\xi,\bar\xi,\xi,\bar\xi)$ commutes only with elements in $C(h_X(i))\cong\U(2)$, if $\xi \neq \pm 1$. Recall that $\U(2)$ is compact. Moreover
$$\diag(\xi,\xi,\bar\xi,\bar\xi) =\diag(\xi,\bar\xi,\xi,\bar\xi)$$
for $\xi = \pm 1$. Again we use the fact that $\Hg(\sX)_{\R}$ cannot be compact and conclude that $Z =\diag(\xi,\xi,\bar\xi,\bar\xi)$.
\end{proof}

Since
$$h_X(\xi) \in Z(\Hg(\sX)) \Rightarrow \diag(\xi,\xi,\bar\xi,\bar\xi) = \diag(\xi^3,\xi,\bar\xi,\bar\xi^3) \Leftrightarrow\xi^3 = \xi \Leftrightarrow\xi^2 = 1 \Leftrightarrow\xi = \pm 1$$
and $h_X(-1) = -E_4$, one concludes from the previous proposition:

\begin{corollary} \label{lulu}
The kernel of the representation $\ad\circ h$ consists of $\{\pm 1\}$.
\end{corollary}

\begin{corollary} \label{uzi}
One has $\Hg(\sX)_{\R}= C(h_G(i))$, if and only if $\Hg(\sX)_{\R}$ has a nondiscrete center.
\end{corollary}
\begin{proof}
Due to the fact that $C(h_G(i))\cong \U(1,1)$ has a nondiscrete center, it is clear that $\Hg(\sX)_{\R}$ has a nondiscrete center, if $\Hg(\sX)_{\R}= C(h_G(i))$. Conversely, if the center $Z(\Hg(\sX)_{\R})$ is nondiscrete, $\dim Z(\Hg(\sX)_{\R})\geq 1$. Moreover the $\R$-valued points of $Z(\Hg(\sX)_{\R})$ are a subgroup of the group of diagonal matrices $\diag(\xi,\xi,\bar\xi,\bar\xi)$ for $|\xi| = 1$ with respect to the basis $\{v_{3,0},v_{2,1},v_{1,2},v_{0,3}\}$ (see Proposition $\ref{cent}$). Since the latter group is given by the onedimensional group $h_G(S^1)(\R)$, one concludes that $Z(\Hg(\sX)_{\R}) \supseteq h_G(S^1)$. Thus $\Hg(\sX)_{\R}\subseteq C(h_G(S^1))$. Recall that reductive groups are almost direct products of their centers and their derived subgroups (see $\ref{redi}$). Moreover note that $\Hg(\sX)_{\R}$ cannot commutative. Otherwise it would be a subgroup of the compact torus
$$C(h(S^1))\cong S^1\times S^1$$
(compare Remark $1.2$), which contradicts the fact that $\Hg(\sX)_{\R}$ cannot be compact (see Lemma $\ref{ccp}$). Thus $\Hg(\sX)_{\R}$ has a nontrivial derived subgroup. Due to the fact that
$$C^{\der}(h_G(S^1)) = C^{\der}(h_G(i)) \cong \SU(1,1)$$
contains no semisimple proper subgroup and does not contain $h_G(S^1)$, one concludes $\Hg(\sX)_{\R}= C(h_G(i))$.
\end{proof}

\begin{proposition} \label{ShiDa}
The pair $(\MT(\sX),h_X)$ is a Shimura datum, if $D\cong \B_1$.
\end{proposition}
\begin{proof}
By our previous results and assumptions,
$$\MT^{\ad}(\sX)_{\R} =\Hg^{\ad}(\sX)_{\R}\cong \PU(1,1).$$
Thus $\MT^{\ad}(\sX)$ is simple and noncompact. Moreover $\ad(h(i))$ yields a Cartan involution (see Lemma $\ref{carty}$). Due to the fact that the conjugation by a diagonal matrix $\diag(a,\ldots,a)$ is the identity map, the weight homomorphism of the Hodge structure $\ad_{\MT(\sX)_{\R}} \circ h$ is given by $\G_{m,\R} \to \{e\}$. Thus the Hodge structure $\ad_{\MT(\sX)_{\R}} \circ h$ has weight zero and all characters of the representation $\ad_{\MT(\sX)_{\R}} \circ h$ are given by $(z/\bar z)^k$ with $k \in \Z$. By Corollary $\ref{lulu}$, the kernel of $\ad\circ h|_{S^1}$ consists of $\{\pm 1\}$. Since $\dim (\MT^{\ad}(\sX)_{\R}) = 3$, this implies that the representation $\ad_{\MT(\sX)_{\R}} \circ h$ is a Hodge structure of type $(1,-1),(0,0),(-1,1)$. Thus we have a Shimura datum as claimed.
\end{proof}

The variation $\sV$ of weight 3 Hodge structures of a nonisotrivial family $\sY \to \sZ$ of Calabi-Yau 3-manifolds has an underlying local system $\sV_{\Z}$ corresponding to an up to conjugation unique monodromy representation
$$\rho:\pi_1(\sZ,z) \to \GL(H^3(\sY_z,\Z)).$$
Let $\sY_z \cong X$. The algebraic group $\Mon^0(\sY)$ denotes the connected component of identity of the Zariski closure of $\rho(\pi_1(\sZ,z))$ in $\GL(H^3(X,\Q))$. The group
$\Mon^0(\sY)$ is a normal subgroup of $\MT^{\der}(\sY)$, if $\sZ$ is a connected complex algebraic manifold (see \cite{Mo}, Theorem $1.4$).
Since $\MT^{\der}(\sY)=\Hg^{\der}(\sY)$ (follows from \cite{JCR}, Corollary $1.3.19$) and $\Sp(H^3(X,\Q),Q)$ is simple, one concludes:

\begin{proposition}
If $\sV_{\Z}$ has an infinite monodromy group, $\sZ$ is a connected complex algebraic manifold, $\sY_z \cong X$ and
$$\Hg(\sY) = \Sp(H^3(X,\Q),Q),$$ one has also
$$\Mon^0(\sY) = \Sp(H^3(X,\Q),Q).$$
\end{proposition}

Consider the Kuranishi family $\sX\to B$ of $X$ and the period map
$$p:B \to {\rm Grass}(H^3(X,\C),b_3(X)/2)$$
associating to each $b \in B$ the subspace
$$F^2(H^3(\sX_b,\C))\subset H^3(\sX_b,\C) \cong H^3(\sX_B,\C) \cong H^3(X,\C)$$
as described in \cite{Voi}, Chapter 10. We say that $F^2(\sH^3)_B$ is constant, if the period map $p:B \to {\rm Grass}(H^3(X,\C),b_3(X)/2)$ is constant. Moreover recall that $\sY \to \sZ$ is a maximal family of Calabi-Yau 3-manifolds, if $\sZ$ can be covered by open subsets $U$ such that each $\sY_U$ is isomorphic to a Kuranishi family.

\begin{Theorem} \label{semifin}
Assume that $\sZ$ is a connected complex algebraic manifold and $f:\sY \to \sZ$ is a maximal family of Calabi-Yau 3-manifolds with $\sY_z \cong X$ and an infinite monodromy group. Then the following statements are equivalent:
\begin{enumerate}
\item \label{mon1} One has that $F^2(\sH^3)_B$ is constant.
\item \label{mon2} The monodromy representation $\rho$ of  $R^3f_*\Q$ satisfies
$$\rho(\gamma)(F^2(H^3(X,\C))) = F^2(H^3(X,\C)) \ \ (\forall \gamma \in \pi_1(\sZ,z)).$$
\item \label{mon3} One has $$\Hg(\sY)_{\R} = C(h_G(i)).$$
\end{enumerate}
\end{Theorem}
\begin{proof}
In \cite{JCR3}, Section 2, we have seen that $\eqref{mon1}$ implies $\eqref{mon2}$.

In the case of $\eqref{mon2}$ we assume that
$$\rho(\gamma)(F^2(H^3(X,\C))) = F^2(H^3(X,\C)) \ \  \mbox{and} \ \ \rho(\gamma)(H^3(X,\R)) = H^3(X,\R) \ \ (\forall \gamma \in \pi_1(\sZ,z)).$$
Hence one has also that
$$\rho(\gamma)(\overline{F^2(H^3(X,\C))}) = \overline{F^2(H^3(X,\C))} \ \ (\forall \gamma \in \pi_1(\sZ,z)).$$
Thus one concludes that $h_G(S^1)$ commutes with $\Mon^0(\sY)$. Hence $\Mon^0(\sY)_{\R}$ is a semisimple group contained in the simple group $C^{\der}(h_G(i))\cong \SU(1,1)$. This implies that $C^{\der}(h_G(i)) = \Mon^0(\sY)_{\R}$. Since $\Hg^{\ad}(\sY) = \Hg^{\ad}(\sX)$ is simple by Theorem $\ref{tzu}$, we conclude
$$C^{\der}(h_G(i)) = \Mon^0(\sY)_{\R} = \Hg^{\der}(\sX)_{\R}$$
from the fact that $\Mon^0(\sY)_{\R}$ is a normal subgroup of $\Hg^{\der}(\sX)_{\R}$. Due to the fact that $h(S^1)$ is not contained in $C^{\der}(h_G(i))$, the reductive group $\Hg(\sX)_{\R}$ has a nontrivial center. Thus from Corollary $\ref{uzi}$, we conclude $\eqref{mon3}$.

Now assume that $\Hg(\sX)_{\R} = C(h_G(i))$. In this case $h_G(i)$ commutes with the elements of $h_{b}(S^1)(\R)$ for each $b \in B$. Hence $h_G(S^1)$ is contained in $C(h_{b}(S^1))$. Due to the fact that $C(h_{b}(S^1))$ contains only the complex structures $\pm h_W(i)_{b}$ and $\pm h_G(i)_{b}$ (see Remark $\ref{1.2}$), one concludes $h_G(i) =h_G(i)_{b}$ from the fact that the $VHS$ is continuous. In other terms $F^2(\sH^3)_B$ is constant.
\end{proof}

\begin{example}
We consider an example, which occurs in \cite{JCR}, $11.3.11$. Let $\sE\to \bP^1\setminus\{0,1,\infty\}$ denote the family of elliptic curves
$$\bP^2 \supset V(y^2z-x(x-z)(x-\lambda z)) \to \lambda \in \bP^1\setminus\{0,1,\infty\}$$
with involution $\iota_{\sE}$ given by $y \to -y$ over $\bP^1\setminus\{0,1,\infty\}$.
Moreover there is a $K3$ surface $S$ with involution $\iota_S$ such that
$$\iota_S|_{H^{1,1}(S)} = \id \ \ \mbox{and} \ \ \iota_S|_{H^{2,0}(S)\oplus H^{0,2}(S)} = -\id.$$
By blowing up the singular sections of the family $\sE \times S/\langle(\iota_{\sE}, \iota_S)\rangle$ over $\bP^1\setminus\{0,1,\infty\}$, one obtains a family $\sY$ of Calabi-Yau 3-manifolds. The Hodge numbers are given by $h^{1,1}=61$ and $h^{2,1} = 1$.

It is a well-known fact that the family $\sE$ has a locally injective period map to the upper half plane. By \cite{JCR}, Example $1.6.9$,
$$F^3(H^3(\sY_{\lambda},\C)) = H^{2,0}(S)\otimes H^{1,0}(\sE_{\lambda}) \ \ \mbox{and} \ \ F^2(H^3(\sY_{\lambda},\C)) = H^{2,0}(S)\otimes H^1(\sE_{\lambda},\C).$$
Thus the $F^2$-bundle in the $VHS$ of $\sY$ is constant and one concludes that $\sY$ a maximal family from the fact that the period map associated with the $F^3$-bundle is locally injective. By Theorem $\ref{semifin}$, one concludes $\Hg(\sY)_{\R}=C(h_G(i))$.
\end{example}

\begin{remark} \label{vu}
For the proof that $\eqref{mon3} \Rightarrow \eqref{mon1}$ in Theorem $\ref{semifin}$ one does not need the assumption that the base is algebraic. It is sufficient to consider the local universal deformation. Thus from \cite{JCR3}, Section 2 one concludes that $X$ cannot occur as a fiber of a family with maximally unipotent monodromy, if $\Hg(\sX)_{\R} = C(h_G(i))$.
\end{remark}

\begin{example}
In \cite{JCR3} one finds an example of a Calabi-Yau 3-manifold $X$ with Hodge numbers $h^{2,1}(X) = 1$ and $h^{1,1}(X)=73$. The manifold $X$ has an automorphism $\alpha$ of degree 3, which extends to an automorphism of $\sX$ over $B$ and acts by a primitive cubic root of unity on $F^2(H^3(X,\C))$. Since $\alpha$ yields an isometry of the Hodge structure of each fiber, the generic Hodge group is contained in the centralizer $C(\alpha)$ of $\alpha$ in $\Sp(H^3(X,\Q),Q)$. By \cite{JCR3}, Lemma $3.4$, one has a description of $C(\alpha)_{\R}$ coinciding with the description of $C(h_G(i))$ in Proposition $\ref{huhu}$. Hence $C(\alpha)_{\R}=C(h_G(i))$. Due to the fact that $C^{\der}(h_G(i))$ does not contain any proper simple subgroup and $\Hg^{\der}(\sX)_{\R}$ is a nontrivial simple subgroup of $C^{\der}(h_G(i))$, one concludes $\Hg(\sX)_{\R}=C(h_G(i))$.
\end{example}

\section{The third case}

Recall that $K$ denotes a maximal compact subgroup of $\Hg(\sX)_{\R}$ and that
$$D = \Hg^{\ad}(\sX)(\R)/\ad(K(\R))$$
is a Hermitian symmetric domain (see Proposition $\ref{juhuu}$). For $D = \B_1$ we have seen that $\Hg(\sX)_{\R}\cong C(h_G(i))$, if and only if $\Hg(\sX)$ has a nondiscrete center (see Corollary $\ref{uzi}$). In Section 2 we have seen that
$$\Hg^{\ad}(\sX) = \Sp^{\ad}(H^3(X,\Q),Q) \ \ \mbox{or} \ \ \Hg^{\ad}(\sX)_{\R} = \PU(1,1).$$
It remains to consider the third case that $\Hg(\sX)$ has a discrete center and $D\cong\B_1$. Thus assume that $\Hg(\sX)$ is simple and has dimension 3. We will study $\Hg(\sX)_{\R}$ by computing its Lie algebra in this case. Let us start with the following observation:

Recall that $\GSp(H^3(X,\R),Q)$ is given by the matrices $M\in H^3(X,\R)$ with
$$M^tQM = rQ \ \ \mbox{for some} \ \ r \in \R.$$
Moreover recall that each representation of $\BS$ on a real vector space $V$ is a Hodge structure by the decomposition of $V_{\C}$ into the eigenspaces with respect to the characters $z^p\bar z^q$ for $p, q \in \Z$ (see \cite{Dat}, $1.1.1$). The conjugation by each diagonal matrix $\diag(a,a,a,a) \in h(\BS)(\R)$ fixes each element of $\GSp(H^3(X,\R),Q)$. Thus the weight homomorphism 
$$\ad_{\GSp(H^3(X,\R),Q)} \circ h \circ w$$
is given by $\G_{m,\R}\to \{e\}$
and the Hodge structure $\ad_{\GSp(H^3(X,\R),Q)} \circ h$ is of weight zero. Therefore the algebra ${\rm Lie}(\GSp(H^3(X,\R),Q))_{\C}$ decomposes into eigenspaces with respect to the characters $(z/\bar z)^k$ for $k \in \Z$.

\begin{pkt} \label{lili}
Now we compute the eigenspace decomposition of ${\rm Lie}(\Sp(H^3(X,\R),Q))$ with respect to the representation $(\ad_{\Sp(H^3(X,\R),Q)}\circ h_X)$ of $S^1$. This description is obtained from the following facts: Each of the following 3-dimensional subgroups of $\Sp(H^3(X,\R),Q)$ given with respect to the basis $\{v_{3,0},v_{2,1},v_{1,2},v_{0,3}\}$ contains an 1-dimensional subgroup on which $h(S^1)$ acts trivially by conjugation. Moreover the kernel of the respective restricted adjoint representation on the respective Lie algebra can be obtained from the description of the conjugation by elements of $h(S^1)$ in Remark $\ref{blbl}$. This allows us to determine the characters of the respective restricted adjoint representation, since we have only characters of the type $(z/\bar z)^k$ for $k \in \Z$ as we have seen above. Since $$10 = \dim \Sp(H^3(X,\R),Q),$$ one checks easily that one can find a basis of eigenvectors by the computations below:
\begin{itemize}
 \item The centralizer $C(h(S^1))$ is a 2-dimensional torus (see Remark $\ref{1.2}$), which yields a corresponding 2-dimensional eigenspace with character 1.
 \item The group $C^{\der}(h_W(i))$ is given by the matrices
$$M =\left(\begin{array}{cccc}\alpha & 0 & \beta & 0\\
0 & \bar \alpha & 0 & - \beta \\
-\bar \beta  & 0 & \bar \alpha & 0 \\
0 & \bar\beta & 0 & \alpha \end{array}\right) \ \ \mbox{with} \ \ |\alpha|^2-|\beta|^2 = 1$$
(this follows from Proposition $\ref{uhuh}$ and Remark $\ref{2u}$). The complexified Lie algebra of $C^{\der}(h_W(i))$ has an eigenspace with character $(\bar z/z)^2$ and an eigenspace with character $(z/\bar z)^2$.
 \item The group $C^{\der}(h_G(i))$ is given by the matrices
$$M =\left(\begin{array}{cccc}\alpha & \beta & 0 & 0\\
\bar \beta & \bar \alpha & 0 & 0 \\
0 & 0 & \bar \alpha & \bar \beta \\
0 & 0 & \beta & \alpha \end{array}\right) \ \ \mbox{with} \ \ |\alpha|^2-|\beta|^2 = 1$$
(this follows from Proposition $\ref{huhu}$ and Remark $\ref{2u}$). The complexified Lie algebra of $C^{\der}(h_G(i))$ has an eigenspace with character $\bar z/z$ and an eigenspace with character $z/\bar z$.
 \item By explicit computations using the definition of $Q$ (see $\eqref{ququ}$), one can easily check that the group $CG$ given by the matrices
$$M =\left(\begin{array}{cccc}
1 & 0 & 0 & 0\\
0 & \alpha & \beta & 0 \\
0 & \bar \beta & \bar \alpha & 0\\
0 & 0 & 0 & 1 \end{array}\right) \ \ \mbox{with} \ \ \det(M) = 1$$
is a subgroup of $\Sp(H^3(X,\R),Q)$. The complexified Lie algebra of the group $CG$ has an eigenspace with character $\bar z/z$ and an eigenspace with character $z/\bar z$.
 \item By explicit computations using the definition of $Q$ (see $\eqref{ququ}$), one can easily check that the group given by the matrices
$$M =\left(\begin{array}{cccc}
\alpha & 0 & 0 & \beta\\
0 & 1 & 0 & 0 \\
0 & 0 & 1 & 0\\
\bar \beta & 0 & 0 & \bar \alpha \end{array}\right) \ \ \mbox{with} \ \ \det(M) = 1$$
is a subgroup of $\Sp(H^3(X,\R),Q)$. The complexified Lie algebra of this group has an eigenspace with character $(\bar z/z)^3$ and an eigenspace with character $(z/\bar z)^3$.
\end{itemize}
\end{pkt}

From now on we make computations with respect to the basis $\{v_{3,0},v_{2,1},v_{1,2},v_{0,3}\}$. The Lie algebra of $\Hg(\sX)_{\R}$ contains clearly the vector space
$${\rm Lie}(h_X(S^1)) = {\rm Span}_{\R}(\diag(3i,i,-i,-3i)).$$
Recall that the representation $\ad \circ h_X$ of $S^1$ on ${\rm Lie}(\Hg(\sX))$ is a weight zero Hodge structure of type $(1,-1),(0,0),(-1,1)$ (follows from Proposition $\ref{ShiDa}$) and the maximal torus of the 3-dimensional simple group $\Hg(\sX)_{\R}$ has dimension 1. The direct sum of the eigenspaces with the characters 1, $z/\bar z$ and $\bar z/z$ coincides with
$${\rm Lie}(C^{\der}(h_G(i)))_{\C}\oplus{\rm Lie}(CG)_{\C}$$
as one concludes from $\ref{lili}$. Hence
$${\rm Lie}(\Hg(\sX))\subset {\rm Lie}(C^{\der}(h_G(i)))\oplus{\rm Lie}(CG).$$
Moreover recall that ${\rm Lie}(\Hg(\sX)_{\R}) \cong  \fsu(1,1)$, where 
$$\fsu(1,1) = {\rm Span}_{\R}(H,X,Y) \ \ \mbox{for} \ \ H = \left(\begin{array}{cc}
i & 0\\
0 & -i \end{array}\right), \ \ X = \left(\begin{array}{cc}
0 & 1\\
1 & 0 \end{array}\right), \ \ Y= \left(\begin{array}{cc}
0 & i\\
-i & 0 \end{array}\right)$$
(compare Remark $\ref{lisu}$) and $H$ generates the Lie subalgebra of a maximal torus of $\Hg(\sX)_{\R}$ with respect to the identification above. Thus ${\rm Span}(H) =  {\rm Lie}(h_X(S^1))$. Since
$$[H,X-iY] = 2Y+2iX = 2i(X-iY) \ \ \mbox{and} \ \ [H,X+iY] = 2Y-2iX = -2i(X-iY),$$
the vector space ${\rm Span}_{\C}(X,Y)$ has a basis of eigenvectors with respect to $\ad(H)$. Therefore each $M\in {\rm Span}_{\R}(X,Y)\subset {\rm Lie}(\Hg(\sX))$ has the form
$$M = \left(\begin{array}{cccc}
0 & * & 0 & 0\\
* & 0 & * & 0 \\
0 & * & 0 & *\\
0 & 0 & * & 0 \end{array}\right)\in {\rm Lie}(C^{\der}(h_G(i)))+{\rm Lie}(CG),$$
where $CG$ was introduced in $\ref{lili}$. The explicit descriptions of $C^{\der}(h_G(i))$ and $CG$ in $\ref{lili}$ and the explicit description of $\SU(1,1)$ in Remark $\ref{2u}$, yield natural isomorphisms
$$C^{\der}(h_G(i))\cong CG\cong \SU(1,1).$$
Thus from the explicit description of $\fsu(1,1)$ in Remark $\ref{lisu}$, we conclude
$$M = \left(\begin{array}{cccc}
0 & x & 0 & 0\\
\bar x & 0 & y & 0 \\
0 & \bar y & 0 & x\\
0 & 0 & \bar x & 0 \end{array}\right)$$
for some $x,y\in \C$. One has an $M \in {\rm Lie}(\Hg(\sX))$ with $x \neq 0$. Otherwise one would have
$$N_1 = \left(\begin{array}{cccc}
0 & 0 & 0 & 0\\
0 & 0 & 1 & 0 \\
0 & 1 & 0 & 0\\
0 & 0 & 0 & 0 \end{array}\right), \ \ N_2 = \left(\begin{array}{cccc}
0 & 0 & 0 & 0\\
0 & 0 & i & 0 \\
0 & -i & 0 & 0\\
0 & 0 & 0 & 0 \end{array}\right) \in {\rm Lie}(\Hg(\sX)),$$
since $\dim {\rm Span}_{\R}(X,Y)=2$. This implies
$$[N_1,N_2] = \diag(0,-2i, 2i,0)\neq 0.$$
But this cannot hold true, since ${\rm Span}_{\R}(\diag(3i,i,-i,-3i))$ is the subvector space of diagonal matrices in ${\rm Lie}(\Hg(\sX)_{\R})$. Moreover one has
{\small $$\left[\left(\begin{array}{cccc}
0 & x & 0 & 0\\
\bar x & 0 & y & 0 \\
0 & \bar y & 0 & x\\
0 & 0 & \bar x & 0 \end{array}\right),
\left(\begin{array}{cccc}
0 & 0 & 0 & 0\\
0 & 0 & z & 0 \\
0 & \bar z & 0 & 0\\
0 & 0 & 0 & 0 \end{array}\right)\right] =
\left(\begin{array}{cccc}
0 & 0 & xz & 0\\
0 & y\bar z-z\bar y & 0 & -zx \\
-\bar x \bar z & 0 & \bar yz-\bar zy & 0\\
0 & \bar x\bar z & 0 & 0 \end{array}\right)\notin {\rm Lie}(\Hg(\sX))$$}
for $x,z \neq 0$. Hence we conclude:

\begin{proposition}
Assume that $\Hg^{\ad}(\sX)_{\R}\cong \PU(1,1)$ and $\Hg(\sX)$ has a discrete center. Then for some $x,y\in \C$ we have
$${\rm Lie}(\Hg(\sX))={\rm Span}_{\R}(\left(\begin{array}{cccc}
3i & 0 & 0 & 0\\
0 & i & 0 & 0 \\
0 & 0 & -i & 0\\
0 & 0 & 0 & -3i \end{array}\right),
\left(\begin{array}{cccc}
0 & 1 & 0 & 0\\
1 & 0 & x & 0 \\
0 & \bar x & 0 & 1\\
0 & 0 & 1 & 0 \end{array}\right),
\left(\begin{array}{cccc}
0 & i & 0 & 0\\
-i & 0 & y & 0 \\
0 & \bar y & 0 & i\\
0 & 0 & -i & 0 \end{array}\right)).$$
\end{proposition}

Now we determine the possible choices of $x,y\in \C$:
{\footnotesize
\begin{equation}\label{gl}\left[\left(\begin{array}{cccc}
0 & i & 0 & 0\\
-i & 0 & y & 0 \\
0 & \bar y & 0 & i\\
0 & 0 & -i & 0 \end{array}\right),
\left(\begin{array}{cccc}
0 & 1 & 0 & 0\\
1 & 0 & x & 0 \\
0 & \bar x & 0 & 1\\
0 & 0 & 1 & 0 \end{array}\right)\right]=\left(\begin{array}{cccc}
2i & 0 & ix-y & 0\\
0 & -2i +\bar xy-x\bar y& 0 & y-ix \\
\bar y+i\bar x & 0 & 2i+x\bar y-\bar xy & 0\\
0 & -i\bar x-\bar y & 0 & -2i \end{array}\right)\end{equation}}
Hence one obtains
$$ix-y=0 \Leftrightarrow ix = y\Leftrightarrow \Im(y)=\Re(x), \ \ \Re(y) = -\Im(x).$$
Thus the matrix on the right hand side of $\eqref{gl}$ is contained in ${\rm Span}(\diag(3i,i,-i,-3i))$ and for the second entry in the second column we obtain
$$-2i +\bar xy-x\bar y=\frac{2}{3}i \Rightarrow \bar xy-x\bar y=\frac{8}{3}i.$$
We have independent of the choice of $x$ and $y$ that
$$\Re(\bar xy-x\bar y) = \Re(\bar xy-\overline{\bar x y}) = 0$$
The previous equations imply:
$$\frac{8}{3} = \Im(\bar xy-x\bar y) = -\Im(x)\Re(y)+\Re(x)\Im(y)+\Re(x)\Im(y)-\Im(x)\Re(y)$$
$$=2\Re(x)^2+2\Im(x)^2 = 2|x|^2$$
By using $ix = y$, we compute
$$\left[\left(\begin{array}{cccc}
0 & i & 0 & 0\\
-i & 0 & y & 0 \\
0 & \bar y & 0 & i\\
0 & 0 & -i & 0 \end{array}\right),
\left(\begin{array}{cccc}
3i & 0 & 0 & 0\\
0 & i & 0 & 0 \\
0 & 0 & -1 & 0\\
0 & 0 & 0 & -3i \end{array}\right)\right]
=\left(\begin{array}{cccc}
0 & 2 & 0 & 0\\
2 & 0 & 2x & 0 \\
0 & 2\bar x & 0 & 2\\
0 & 0 & 2 & 0 \end{array}\right)$$
$$\mbox{and} \ \
\left[\left(\begin{array}{cccc}
3i & 0 & 0 & 0\\
0 & i & 0 & 0 \\
0 & 0 & -1 & 0\\
0 & 0 & 0 & -3i \end{array}\right)
,\left(\begin{array}{cccc}
0 & 1 & 0 & 0\\
1 & 0 & x & 0 \\
0 & \bar x & 0 & 1\\
0 & 0 & 1 & 0 \end{array}\right)\right]
=\left(\begin{array}{cccc}
0 & 2i & 0 & 0\\
-2i & 0 & 2y & 0 \\
0 & 2\bar y & 0 & 2i\\
0 & 0 & -2i & 0 \end{array}\right).$$
Note that each connected simple Lie group of $\GL_n(\R)$ is also the connected component of identity of the Lie group of real valued points of an $\R$-algebraic group (this follows from \cite{Sat}, {\bf I}. Proposition 3.6). Thus we conclude:

\begin{proposition}
For each $x\in \C$ with $|x|=\frac{2}{\sqrt{3}}$ there is a simple $\R$-algebraic subgroup $$G_{x}\subset \Sp(H^3(X,\R),Q)$$ of dimension 3 such that $h(S^1)\subset G_{x}$.
\end{proposition}

\begin{lemma} \label{maxmon}
Each unipotent matrix in $G_{x}$ has a Jordan block of length $\geq 3$.
\end{lemma}
\begin{proof}
A unipotent matrix in $G_{x}$, whose Jordan blocks have the maximal length 2, would correspond to a matrix $M\in{\rm Lie}(G_x)$, whose square is zero. One has that
$$(m_{i,j})=M^2=\left(\begin{array}{cccc}
a3i & c+bi & 0 & 0\\
c-bi & ai & cx+by & 0 \\
0 & c\bar x+ b\bar y & -ai & c+bi\\
0 & 0 & c-bi & -3ai \end{array}\right)^2 = 0$$
with $a,b,c\in \R$ is satisfied, only if
$$m_{1,2}=4ai(c+bi) = 0.$$
Hence $a = 0$ or $c+bi=0$. The reader checks easily that $M^2$ cannot be zero in either case with the exception given by $M=0$.
\end{proof}

\begin{example}
In \cite{DM2} there is a list of explicitly computed examples of variations of Hodge structures of families $\sY\to \bP^1\setminus\{0,1,\infty\}$ of Calabi-Yau 3-manifolds with 1-dimensional complex moduli. Note that each of these variations has a monodromy group containing a unipotent matrix, which has only Jordan blocks of length $\leq 2$. Due to the fact that $\Mon^0(\sY) \subseteq\Hg(\sY)$, we conclude from Lemma $\ref{maxmon}$ that there is no $x$ with $|x|=\frac{2}{\sqrt{3}}$ such that $\Hg(\sY)_{\R}\cong G_x$. Moreover each example in \cite{DM2} has maximally unipotent monodromy. Thus we are not in the case $\Hg(\sY)_{\R}= C(h_G(i))$ for these examples. Therefore the examples of \cite{DM2} have a generic Hodge group given by $\Sp(H^3(Y,\Q),Q)$, where $Y$ denotes an arbitrary fiber of the respective family $\sY$.
\end{example}

It would be very nice to find an example for the third case $\Hg(\sX)_{\R}=G_x$. At present there is no example of a family of Calabi-Yau manifolds with 1-dimensional complex moduli known to the author, which satisfies the third case. Nevertheless one finds a Calabi-Yau like variation of Hodge structures of third case, which arises in a natural way over a curve as we will see now (for the definition of a Calabi-Yau like $VHS$ of third case see $\ref{cyarhg}$). For this example one uses the construction of C. Borcea \cite{Bc}:

\begin{construction} \label{bobo}
Let $E_1,E_2,E_3$ be elliptic curves with involutions $\iota_1,\iota_2,\iota_3$ such that $E_j/\iota_j\cong \bP^1$. The singular variety
$$E_1 \times E_2 \times E_3/\langle(\iota_1,\iota_2),(\iota_2,\iota_3)\rangle$$
yields a Calabi-Yau 3-manifold $C$ by blowing up the singularities. The isomorphism class of $C$ depends on the choice of the sequence of blowing ups. Nevertheless the Hodge structure on $H^3(C,\Z)$ does not depend on the choice of this sequence and is given by the tensor product
$$H^3(C,\C) = H^1(E_1,\C)\otimes H^1(E_2,\C)\otimes H^1(E_3,\C)$$
of the respective Hodge structures.

Let $f_1:\sE \to \A^1\setminus\{0,1\}$ denote the family of elliptic curves given by
$$\bP^2\supset V(y^2z= x(x-z)(x-\lambda z)) \to \lambda\in \A^1\setminus\{0,1\}.$$
By using the involution of $\sE$ over $\A^1\setminus\{0,1\}$ and three copies of $\sE \to \A^1\setminus\{0,1\}$, one can give a relative version of the previous construction. Let $f_3:\sC\to (\A^1\setminus\{0,1\})^3$ denote a family obtained by this relative version of C. Borcea's construction.
\end{construction}

Recall that a Calabi-Yau 3-manifold $X$ has complex multiplication ($CM$), if the Hodge group $\Hg(H^3(X,\Q),h)$ is a torus. For $\Hg(\sX)_{\R}=C(h_G(i))$ the pair is a Shimura datum (see Proposition $\ref{ShiDa}$). Thus we have a dense set of $CM$ fibers.\footnote{The proof uses arguments, which occur already in \cite{JCR3}, Section 3. One has only to replace $C(\alpha)$ by $\Hg(\sX)$ and use the same arguments, which occur after the proof of \cite{JCR3}, Lemma $3.4$.} But in this case one cannot have maximally unipotent monodromy (see Remark $\ref{vu}$). Moreover the associated Hermitian symmetric domain has a dimension larger than the dimension of the basis for $\Hg(\sX)=\Sp(H^3(X,\Q),Q)$. For this case one conjectures that only finitely many $CM$ fibers occur. Hence for families of Calabi-Yau 3-manifolds with onedimensional complex moduli it is feasible to conjecture that the existence of infinitely many nonisomorphic $CM$ fibers and maximally unipotent monodromy exclude each other. This does not hold true for Calabi-Yau 3-manifolds with higher dimensional complex moduli, since the family $f_3:\sC\to (\A^1\setminus\{0,1\})^3$ has maximally unipotent monodromy and a dense set of $CM$ fibers:

\begin{remark}
Let $\Delta^*$ denote the punctured disc. One finds a neighbourhood $U$ of the point $(0,0,0)\in \A^3$ such that $\sC$ is locally defined over $(\Delta^*)^3\subset U$. Let $D_1,D_2,D_3$ denote the irreducible components of the complement of $(\Delta^*)^3\subset U$ and $\gamma_i$ denote a closed path given by a loop around $D_i$. The family $f_1:\sE \to \A^1\setminus\{0,1\}$ of elliptic curves has unipotent monodromy around $0$ with
$$\rho(\gamma) = \left(\begin{array}{cc}
1 & 2\\
0 & 1
\end{array} \right)$$
with respect to a basis $\{a,b\}$ (this follows from the computations in \cite{JCR}, Section $3.3$). Thus one computes easily that
$$N_{r,s,t} = r\log \rho(\gamma_1)+s\log \rho(\gamma_2)+t\log \rho(\gamma_{3}) = \left(\begin{array}{cccccccc}
0 & 2t & 2s & 0 & 2r & 0 & 0 & 0\\
0 & 0 & 0 & 2s & 0 & 2r & 0 & 0\\
0 & 0 & 0 & 2t & 0 & 0 & 2r & 0\\
0 & 0 & 0 & 0 & 0 & 0 & 0 & 2r\\
0 & 0 & 0 & 0 & 0 & 2t & 2s & 0\\
0 & 0 & 0 & 0 & 0 & 0 & 0 & 2s\\
0 & 0 & 0 & 0 & 0 & 0 & 0 & 2t\\
0 & 0 & 0 & 0 & 0 & 0 & 0 & 0
\end{array} \right)$$
with respect to the basis
$$\sB = \{a_1\otimes a_2\otimes a_3, \ \ a_1\otimes a_2\otimes b_3, \ \ a_1\otimes b_2\otimes a_3, \ \ a_1\otimes b_2\otimes b_3,$$
$$b_1\otimes a_2\otimes a_3, \ \ b_1\otimes a_2\otimes b_3, \ \ b_1\otimes b_2\otimes a_3, \ \ b_1\otimes b_2\otimes b_3\}.$$
By analogue computations, one gets the same result for all maximal-depth normal crossing points of $(\A^1\setminus\{0,1\})^3$. Thus the family $\sC \to (\A^1\setminus\{0,1\})^3$ has maximally unipotent monodromy around each maximal-depth normal crossing point (for the definition of maximally unipotent monodromy see \cite{Mor}). Moreover $C$ has $CM$, if and only if $E_1,E_2,E_3$ have $CM$ as complex tori (see \cite{Bc}, Proposition $3.1$).
Since it is a well-known fact that $\sE$ has a dense set of fibers $\sE_{\lambda}$ such that $\sE_{\lambda}$ has $CM$, one concludes that $\sC$ has a dense set of $CM$ fibers.
\end{remark}

Now we come to the Calabi-Yau like $VHS$ of third type. Let $\Delta \subset (\A^1\setminus\{0,1\})^3$ be the diagonal obtained from the closed embedding
$$\A^1\setminus\{0,1\} \hookrightarrow (\A^1\setminus\{0,1\})^3 \ \ \mbox{via} \ \ x \to (x,x,x).$$
As we will see the rational $VHS$ of the restricted family $\sC_{\Delta} \to \Delta$ contains a sub-$VHS$ of third type. Let
$$\sH^1 = R^1(f_1)_*\Q \otimes \sO_{\Delta} \ \ \mbox{and} \ \ \sH^3 = R^3(f_3|_{\sC_{\Delta}})_*\Q \otimes \sO_{\Delta}.$$

\begin{pkt} \label{4.8}
One has that $\sH^3 =(\sH^1)^{\otimes 3}$ (see also \cite{VZ}, Remark $7.4$) and $F^3(\sH^3)$ is contained in the symmetric product ${\rm Sym}^3(\sH^1)$. Hence
$$H^{3,0}(\sC_{(\lambda,\lambda,\lambda)}), H^{0,3}(\sC_{(\lambda,\lambda,\lambda)}) \subset {\rm Sym}^3(H^1(\sE_{\lambda},\C))$$
for each $(\lambda,\lambda,\lambda) \in \Delta$. Since $F^3(\sH^3)\subset {\rm Sym}^3(\sH^1)$, one obtains $\nabla_t\omega(b) \in {\rm Sym}^3(H^1(\sE_{\lambda},\Q))$ for each section $\omega \in F^3(\sH^3_{\Delta})(U)$ and $t \in T_b\Delta$. By Bryant-Griffiths \cite{BG}, one has that $F^2(\sH^3)$ is generated by the sections of $F^3(\sH^3)$ and their differentials. Therefore one concludes that  $F^2(\sH^3) \cap {\rm Sym}^3(\sH^1)$ is of rank 2 and we have a polarized rational variation $\sV$ of Hodge structures of type
$$(3,0), \ \ (2,1), \ \ (1,2), \ \ (0,3)$$
with the underlying local system ${\rm Sym}^3(R^1(f_1)_*\Q)$ of rank 4. This $VHS$ satisfies that $F^2(\sV)$ is generated by the sections of $F^3(\sV)$ and their differentials along $\Delta$, and that $F^1(\sV)= F^3(\sV)^{\bot}$ with respect to the polarization. By \cite{BG}, these two properties characterize the $VHS$ of a family of Calabi-Yau 3-manifolds. In this sense $\sV$ is a Calabi-Yau like sub-$VHS$ of the rational $VHS$ of $\sC_{\Delta}$.
\end{pkt}

\begin{pkt} \label{cyarhg}
Let $M$ be connected complex manifold and $\sW \to M$ be a Calabi-Yau like $VHS$ with
$$h^{3,0}(\sW_m) = h^{2,1}(\sW_m) = h^{1,2}(\sW_m) =h^{0,3}(\sW_m) = 1$$
for each $m \in M$ in the sense of $\ref{4.8}$. We say that $\sW$ is of third type, if the center of its generic Hodge group is discrete and the associated Hermitian symmetric domain is $\B_1$. Note that all previous arguments are also valid for a Calabi-Yau like $VHS$, which is not necessarily the $VHS$ of a family of Calabi-Yau 3-manifolds. Thus there is an $x \in \C$ with $|x| = \frac{2}{\sqrt{3}}$ such that
$\Hg(\sW)_{\R} = G_x$ for a Calabi-Yau like $VHS$ of third type.
\end{pkt}

Let $E$ be an elliptic curve and $M \in \GL(H^1(E,\Q))$ be given by
$$M = \left(\begin{array}{cc}
a & b\\
c & d
\end{array} \right)\in \GL(H^1(E,\Q))$$
with respect to a basis $\{e_1,e_2\}$ of $H^1(E,\Q)$. Moreover let
$$Kr^3(M)=M\otimes M\otimes M$$
denote the third Kronecker power of $M$. One can easily check that
$$Kr^3(M)({\rm Sym}^3(H^1(E,\Q))) = {\rm Sym}^3(H^1(E,\Q))$$
for each $M \in \GL(H^1(E,\Q))$. Moreover one can easily compute that $Kr(M)$ acts on ${\rm Sym}^3(H^1(E,\Q))$ by the matrix
\begin{equation} \label{hihihi} r(M) =\left(\begin{array}{cccc}
a^3 & 3a^2b & 3ab^2 & b^3 \\
a^2c & a^2d+2abc & 2abd+b^2c & b^2d \\
ac^2 & acd+bc^2 & ad^2+2bcd & bd^2 \\
c^3 & 3c^2d & 3cd^2 & d^3
\end{array} \right)
\end{equation}
with respect to the basis
$$\{e_1\otimes e_1\otimes e_1, \ \ e_1\otimes e_1\otimes e_2 \ \ + \ \ e_1\otimes e_2\otimes e_1 \ \  + \ \ e_2\otimes e_1\otimes e_1,$$
$$e_1\otimes e_2\otimes e_2 \ \ + \ \ e_2\otimes e_1\otimes e_2 \ \ + \ \ e_1\otimes e_2\otimes e_2, \ \ e_2\otimes e_2\otimes e_2\}.$$

\begin{lemma}
One has the homomorphisms
$$r:\GL(H^1(E,\Q)) \to \GL({\rm Sym}^3(H^1(E,\Q)))$$
and
$$r|_{\SL(H^1(E,\Q))}:\SL(H^1(E,\Q)) \to \SL({\rm Sym}^3(H^1(E,\Q)))$$
of $\Q$-algebraic groups.
\end{lemma}
\begin{proof}
From $\eqref{hihihi}$ one concludes that $r$ is an regular map.
Note that the determinant of $r(M)$ is given by $\det^{6}(M)$ for each $M \in \GL(H^1(E,\Q))$. This follows by computing $\det(r(J_M))$, where $J_M$ denotes the associated Jordan form of $M$. Since one can easily check that $Kr^3$ respects the matrix multiplication, one concludes that the same holds true for $r$. Thus we obtain the homomorphisms of $\Q$-algebraic groups as claimed.
\end{proof}

Let $G$ denote the Zariski closure of $r(\SL(H^1(E,\Q)))$ in $\GL({\rm Sym}^3(H^1(E,\Q)))$. It is a well-known fact that $G$ is an algebraic group.

\begin{lemma} \label{hogi}
The group $G$ has at most dimension 3.
\end{lemma}
\begin{proof}
Let
$$M = \left(\begin{array}{cc}
a & b\\
c & d
\end{array} \right)\in \GL(H^1(E,\Q)).$$
For $(m_{i,j})=r(M)$ one has that
$$m_{2,2}^3 = (a(ad+2bc))^3 = a^3(a^3d^3 + 6a^2bcd^2+ 12ab^2c^2d+ 8b^3c^3)$$
$$= m_{1,1}(m_{1,1}m_{4,4}+\frac{2}{3}m_{1,2}m_{4,3}+\frac{4}{3}m_{1,3}m_{4,2}+ 8m_{1,4}m_{4,1})$$
(follows from $\eqref{hihihi}$). In an analogue way one can express $m_{2,3}^3,m_{3,2}^3,m_{3,3}^3$ by equations with entries $m_{i,j}$ such that $\{i,j\}\cap \{1,4\} \neq \emptyset$.
Note that for all other entries $m_{i,j}$ of $r(M)$ such that $\{i,j\}\cap \{1,4\} \neq \emptyset$ the power $m_{i,j}^3$ satisfies some equation in terms of
$$m_{1,1}= a^3,m_{1,4}= b^3,m_{4,1}= c^3,m_{4,4}= d^3$$
(compare $\eqref{hihihi}$). Due to these facts, one finds enough equations such that the Zariski closure $\overline{r(\GL(H^1(E,\Q))}$ of the group $r(\GL(H^1(E,\Q))(\Q)$ has at most dimension 4. Since $\det (r(M))=\det^{6}(M)$, the set on the right hand site of the inequality
$$G^0\subseteq (\overline{r(\GL(H^1(E,\Q))}\cap \SL({\rm Sym}^3H^1(E,\Q)))^0$$
is a proper Zariski closed subset of $\overline{r(\GL(H^1(E,\Q))}^0$. Thus one concludes that
$$\dim G \leq3.$$
\end{proof}

Note that the Hodge structure of $\sC_{(\lambda,\lambda,\lambda)}$ is given by the tensor product $H^1(\sE_{\lambda},\Q)^{\otimes 3}$. Thus the associated representation of $S^1$ is given by $Kr^3\circ h_{\lambda}$, where $h_{\lambda}$ denotes the Hodge structure of $\sE_{\lambda}$. Therefore the sub-Hodge structure on ${\rm Sym}^3(H^1(\sE_{\lambda},\Q))$ is given by
$$h' =r \circ h_{\lambda}.$$
One concludes $h'(S^1) \subset G_{\R}$, since $h_{\lambda}(S^1) \subset \SL(H^1(E,\R))$ and $r$ yields a homomorphism $\SL(H^1(E,\R)) \to G_{\R}$.

\begin{proposition}
The variation $\sV$ of Hodge structures is of third type.
\end{proposition}
\begin{proof}
Since $h'(S^1) \subset G_{\R}$, the conjugation by $h'(i)$ yields a Cartan involution of $G_{\R}$. Thus $G$ is reductive. Since $\dim G \leq 3$, this group is not only reductive, but simple. This follows from the fact that the smallest simple Lie algebras have dimension 3 and $G$ is clearly not commutative. Therefore the center of $G_{\R}$ is discrete the associated hermitian symmetric domain is $\B_1$. Hence $\sV$ is of third type.
\end{proof}

\section*{Acknowledgements}
This paper was written at the Graduiertenkolleg ``Analysis, Geometry and Stringtheory'' at Leibniz Universit\"at Hannover. I would like to thank Klaus Hulek for his interest and his comments, which helped to improve this text. Moreover I would like to thank Martin M\"oller for a fruitful discussion and Eckart Viehweg for the hint to \cite{VZ}, which helped to find a Shimura curve with a Calabi-Yau like $VHS$ of third type.


\begin{thebibliography}{XXX} 
\bibitem{Bc} Borcea, C.: Calabi-Yau threefolds and complex multiplication. Essays on mirror
manifolds. Internat. Press,  Hong Kong (1992) 489-502.
\bibitem{BG} Bryant, R., Griffiths, P.: Some Observations on the Infinitesimal Period Relations for Regular Threefolds with Trivial Canonical Bundle. In: Arithmetic and Geometry {\bf II}, Progress in Mathematics {\bf 36}, Birkh\"auser, Boston Basel Stuttgart (1983) 77-102.
\bibitem{De2} Deligne, P.:  Travaux de Shimura. In: Seminaire Bourbaki,
{\bf 389} (1970/71), Lecture Notes in
Mathematics {\bf 244},  Springer-Verlag,  Berlin (1971)  123-165.
\bibitem{Dat} Deligne, P.:  Vari\'et\'es de Shimura: interpr\'etation modulaire, et techniques de
construction de mod\`eles canoniques. In: Automorphic forms, representations and $L$-functions
(Proc. Sympos. Pure Math., Oregon State Univ., Corvalis, Ore., 1977), Part 2, Proc. Sympos. Pure
Math., XXXIII, AMS, Providence, R.I. (1979) 247-289.
\bibitem{DM2} Doran, C., Morgan, J.: Mirror Symmetry and Integral Variations of Hodge Structure. Stud. in adv. Math. {\bf 38} (2006) 517-538.
\bibitem{Helga} Helgason, S.: Differential Geometry, Lie Groups, and Symmetric Spaces.
Graduate Studies in Mathematics {\bf 34}, AMS (2001) Providence, RI.
\bibitem{Milne} Milne, J. S.: Introduction to Shimura Varieties. (2004) {\em preprint}.
\bibitem{Mor} Morrison, D.: Compactifications of moduli spaces inspired by mirror symmetry. Journ\'ees de
g\'eom\'etrie alg\'ebrique d\'{ }Orsay, Ast\'erisque {\bf 218} (1993) 243-271.
\bibitem{Mo} Moonen, B.: Linearity properties of Shimura varieties. Part {\bf I}, J. Algebraic Geom.
{\bf 7} (1998) 539-567.
\bibitem{JCR} Rohde, J. C.: Cyclic coverings, Calabi-Yau manifolds and Complex multiplication. Lecture Notes in Mathematics {\bf 1975}, Springer-Verlag (2009) Berlin, Heidelberg.
\bibitem{JCR3} Rohde, J. C.: Maximal automorphisms of Calabi-Yau manifolds versus maximally unipotent monodromy. (2009) arXiv:0902.4529v2 {\em to appear in manuscripta mathematica}.
\bibitem{Sat} Satake, I.: Algebraic structures of symmetric domains.  Kan\^o Memorial
Lectures {\bf4}, Iwanami Shoten, Princeton University Press (1980) USA.
\bibitem{VZ} Viehweg, E., Zuo K.: Families over curves with a strictly maximal Higgs field.
Asian J. Math. {\bf 7} (2003) 575-598.
\bibitem{Voi} Voisin, C.: Th\'eorie de Hodge et g\'eom\'etrie alg\'ebrique complexe.
Cours sp\'ecialis\'es {\bf 10}, SMF (2002) France.
\end{thebibliography}
\end{document}